\documentclass[12pt]{amsart} 

\usepackage[margin=2.9cm]{geometry}

\usepackage{amsmath, amsthm, amssymb, color,verbatim,bbm}
\usepackage{hyperref}      
\usepackage{units}
\usepackage{array}
\usepackage{blkarray}
\usepackage{graphicx}
\usepackage{cite}
\usepackage{mathrsfs}
\usepackage{appendix}
\usepackage{makecell}
\usepackage{arydshln}
\usepackage{cancel}
\usepackage{soul}
\usepackage{algorithmic}
\usepackage[section]{algorithm}

\usepackage{tikz}
\usetikzlibrary{cd}
\usetikzlibrary{backgrounds, shapes, positioning}

\newtheorem{theorem}{Theorem}
\numberwithin{theorem}{section}
\newtheorem{proposition}[theorem]{Proposition}
\newtheorem{lemma}[theorem]{Lemma}
\newtheorem{corollary}[theorem]{Corollary}

\theoremstyle{definition}
\newtheorem{definition}[theorem]{Definition}
\newtheorem{example}[theorem]{Example}
\theoremstyle{remark}
\newtheorem{remark}[theorem]{Remark}

\newcommand{\GT}{\mathrm{GT}}

\newcommand{\PD}{\mathrm{PD}}
\newcommand{\RR}{\mathbb{R}}
\newcommand{\CC}{\mathbb{C}}
\newcommand{\ZZ}{\mathbb{Z}}
\newcommand{\KK}{\mathbb{K}}

\newcommand{\capac}{\mathrm{cap}}
\newcommand{\supp}{\mathrm{supp}}

\newcommand{\SO}{\operatorname{SO}}
\newcommand{\SL}{\operatorname{SL}}
\newcommand{\GL}{\operatorname{GL}}

\newcommand{\tr}{\mathrm{tr}}
\newcommand{\cp}{\mathrm{cp}}
\newcommand{\GSL}{G_{\SL}}

\newcommand{\mlt}{\mathrm{mlt}}
\newcommand{\mltb}{\mathrm{mlt_b}}

\newcommand{\T}{^\mathsf{T}}
\newcommand{\St}{\mathrm{Stab}}

\newcommand{\exSymbol}{$\diamondsuit$}

\title[Invariant theory and scaling algorithms for ML estimation]{Invariant theory and scaling algorithms \\ for maximum likelihood estimation}
\author[Carlos Am\'{e}ndola, Kathl\'{e}n  Kohn, Philipp Reichenbach, Anna Seigal]{{\normalsize Carlos Am\'{e}ndola, Kathl\'{e}n  Kohn, Philipp Reichenbach, Anna Seigal}}

\begin{document}

\begingroup
\let\MakeUppercase\relax 
\maketitle
\endgroup
\vspace*{-0.5cm}
\begin{abstract}
We uncover connections between maximum likelihood estimation in statistics and norm minimization over a group orbit in invariant theory. 
We focus on Gaussian transformation families, which include
matrix normal models and Gaussian graphical models given by transitive directed acyclic graphs.
We use stability under group actions to characterize boundedness of the likelihood, and existence and uniqueness of the maximum likelihood estimate. 
Our approach reveals promising consequences of the interplay between invariant theory and statistics. 
In particular, existing scaling algorithms from statistics can be used in invariant theory, and vice versa.
\end{abstract}

\section{Introduction}

The task of fitting data to a model is fundamental in statistics. A statistical model is a set of probability distributions.
We seek a point in a model that best fits some empirical data. 
A widespread approach is to maximize the likelihood of observing the data as we range over the model. A point that maximizes the likelihood is called a \emph{maximum likelihood estimate} (MLE). There are several ways to compute an MLE for different statistical models, usually via optimization approaches that find a local maximum~\cite{MLEoptimization,MLEtutorial}.
There is growing interest in understanding when algorithms to find an MLE are guaranteed to work, and under which conditions an MLE exists or is unique. 
In this paper, we approach such questions using invariant theory. 

Invariant theory studies actions of groups on vector spaces or, more generally, on algebraic varieties.
An important concept is the orbit of a point under the group action, which is the set of all points that differ from the original point by a transformation in the group.
The \emph{capacity} of a point is the infimal norm along its orbit.
If the orbit is closed, the capacity is attained; otherwise the capacity is attained only on the orbit closure.
Points with zero capacity are called unstable;
they form the null cone, a classical object in invariant theory dating back to Hilbert~\cite{hilbert},
which is of particular interest for moduli spaces of algebraic objects.
More recently, algorithmic approaches to stability questions have been taken, with a special focus on testing  null cone membership~\cite{Allen-Zhu,PeterAvi1, PeterAvi3, DerksenMakam, GGOWoperatorScaling, IQS}. 
A number of applied problems have been cast within an invariant theoretic framework,  including questions in quantum information theory, complexity theory and analytic inequalities, see e.g.~\cite[\S1.2]{PeterAvi3}. 

There is a close connection between statistical models and group actions, dating back to Fisher \cite{fisher1934}. 
We build a bridge between invariant theory and maximum likelihood estimation. 
In this paper, we study this connection in the setting of multivariate Gaussian models.
 We define \emph{Gaussian group models}, multivariate Gaussian models whose concentration matrices are of the form $g\T g$, where $g$ lies in a group. 
 Examples of Gaussian group models are matrix normal models and Gaussian graphical models defined by transitive directed acyclic graphs.

The connection between invariant theory and maximum likelihood estimation also holds for discrete statistical models, as we discuss in our companion paper~\cite{companion}. 
There, we show that 
maximum likelihood estimation in log-linear models is equivalent to computing the capacity under a torus action.
Both Gaussian group models and log-linear models fall within the framework of exponential families.

\subsection*{Main contributions} 
We show that finding the MLE can be cast as the problem of computing the capacity, see
 Propositions~\ref{prop:twoStepGeneral} and~\ref{prop:twoStepGeneralComplex}.
Viewing maximum likelihood estimation as a norm minimization problem allows us to build a correspondence between notions of stability from invariant theory and MLE properties:
\[
\left\{ \begin{matrix}
\text{unstable} \\ 
\text{semistable} \\ 
\text{polystable} \\ 
\text{stable} 
\end{matrix} \right\} 
\qquad \longleftrightarrow \qquad 
\left\{ \begin{matrix}
\text{likelihood unbounded from above} \\ 
\text{likelihood bounded from above} \\ 
\text{MLE exists} \\ 
\text{MLE exists uniquely} 
\end{matrix} \right\} 
\]
For some models we prove an exact equivalence between the four notions of stability on the left and the four properties of the MLE on the right, see Theorem~\ref{thm:MLEversusStabilityComplex} for complex Gaussian group models.
For real statistical models, we prove real analogues of the correspondence that hold at two levels of generality; see Theorems~\ref{thm:gaussianMLEstable} and~\ref{thm:MLEversusStabilityReal}.
The two levels of generality correspond to non-reductive and reductive groups. 

While invariant theory often focuses on reductive groups, Gaussian group models are natural to study in both settings.
For matrix normal models, which are given by reductive groups, 
we use descriptions of the null cone to give improved bounds on the number of samples generically required for a bounded likelihood function, see Theorem~\ref{thm:nullconeFills} and Corollary~\ref{cor:newMLEbound}.
Gaussian models defined by transitive directed acyclic graphs are in general given by non-reductive groups.
For such models, our results translate to exact conditions for MLE existence in terms of linear independence of the rows of the sample matrix, see Theorem~\ref{thm:tdagnullcone}.

Our connection between invariant theory and maximum likelihood estimation leads to the algorithmic consequences that we detail below.

\begin{figure}[htbp]
\centering
\begin{tikzpicture}[
roundnode/.style={ellipse, draw=black, thick, minimum size=10mm},
squarednode/.style={rectangle, draw=black, thick, minimum size=7mm},
description/.style={rectangle, thick, minimum size=5mm},
]
\node[roundnode] (sl){$\GL_{m_1} \times \GL_{m_2}$};
\node[roundnode] (g)[right=4cm of sl]{$G$};
\node[squarednode] (operator)[above=of sl] {operator scaling};
\node[squarednode] (flip) [below =of sl] {flip-flop algorithm};
\node[squarednode] (null) [above=of g] {$\qquad$norm minimization$\qquad$};
\node[squarednode] (ips) [below=of g] {IPS for Gaussian group models};
\node[description] (left) [above=0.3cm of operator] {Left-right action};
\node[description] (torus) [above=0.3cm of null] {General group action};
\node[description] (inv) [left=0.5cm of operator] {Invariant Theory:};
\node[description] (stat) [left=0.5cm of flip] {Statistics:};

\draw[thick] (sl.north) -- (operator.south);
\draw[thick] (sl.south) -- (flip.north);
\draw[thick] (g.south) -- (ips.north);
\draw[thick] (g.north) -- (null.south);
\draw[<->, dashed, thick] (flip.north east) to[bend right] (operator.south east);
\draw[->, thick] (operator.east) -- (null.west);
\draw[->, thick] (flip.east) -- (ips.west);
\end{tikzpicture}

\caption{Overview of different scaling algorithms.  For the invariant theory algorithms, we use matrices of determinant one, e.g. $\SL_{m_1} \times \SL_{m_2} \subseteq \GL_{m_1} 
\times \GL_{m_2}$.}
    \label{fig:algorithms}
\end{figure}
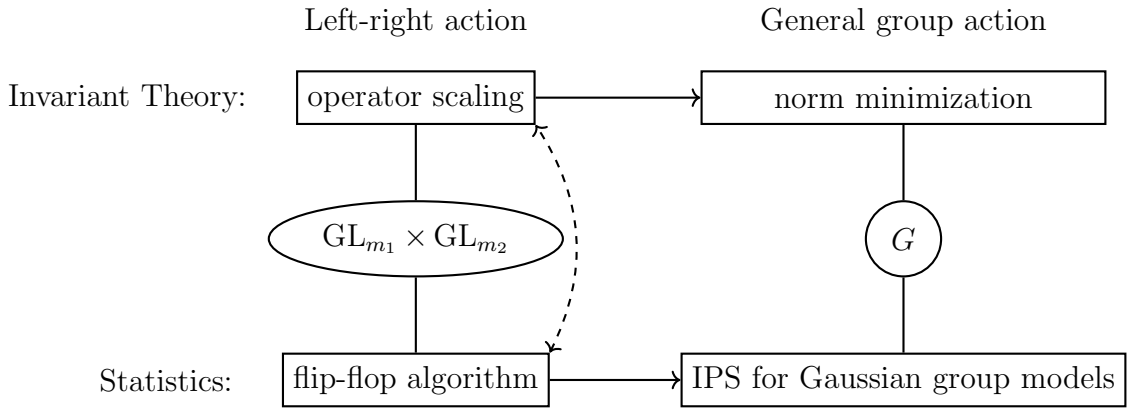

\subsection*{Algorithmic implications}
Scaling algorithms are iterative algorithms existing both in statistics and in invariant theory. They are characterized by update steps, which are given by a group action in many instances.
For matrix normal models, we show the equivalence of two alternating algorithms: 
operator scaling from invariant theory for null cone membership testing \cite{GurvitsOperatorScaling, GGOWoperatorScaling}, and the flip-flop algorithm from statistics for maximum likelihood estimation \cite{dutilleul1999mle,lu2005likelihood}; see the left of Figure~\ref{fig:algorithms} and Section~\ref{sec:matrixnormalalg}.
This equivalence enables us to obtain a 
complexity analysis for the flip-flop algorithm (see Theorem~\ref{thm:flipflopcomplexity}) by directly adapting the result for the corresponding null cone membership problem from~\cite[Theorem~1.1]{PeterAvi1}.

We now describe how this can be extended to more general scaling algorithms, see the right hand side of Figure~\ref{fig:algorithms}.
The flip-flop algorithm can be thought of as an instance of {\em iterative proportional scaling (IPS)} (or iterative proportional fitting (IPF)), a family of methods to find the MLE in a statistical model~\cite{IPFienberg,GaussianIPF}.
For Gaussian group models, 
we can find an MLE via the geodesically convex optimization approaches from~\cite{PeterAvi3} that minimize the norm over an orbit.
These algorithms can be thought of as generalizations of operator scaling. We therefore regard them as IPS for Gaussian group models.
Properties (such as complexity or efficiency) of scaling algorithms for testing stability translate, under our correspondence, to properties of the corresponding IPS algorithm for finding the MLE.

The connection between norm minimization in invariant theory and IPS in statistics is discussed for torus actions and discrete models in our companion paper~\cite{companion}.
There, \cite[Figure~4]{companion} gives the analogue of Figure~\ref{fig:algorithms} for the setting of a discrete model and a torus action (rather than a Gaussian model and a general group action). The starting point of both Figures is Sinkhorn scaling~\cite{sinkhorn1964}, an alternating method that involves the left-right action of a product of two tori. The alternating idea from Sinkhorn's scaling generalizes to products of groups, e.g. to operator scaling and the flip-flop algorithm in Figure~\ref{fig:algorithms}.

We see that algorithms in invariant theory can be used in maximum likelihood estimation, and vice versa. 
In statistics, many iterative algorithms for finding the MLE are well-known. 
It is a more recent question to understand when they converge, i.e. when an MLE exists, and when convergence is to a unique solution, i.e. when the MLE is unique. 
The historical progression is the opposite in invariant theory: the distinction between different types of stability is classical, while more recent approaches use algorithms to test instability. 
Our results are intended to stimulate further research to deepen the connection between the fields.\footnote{Since the preprint of this paper first appeared, our dictionary between ML estimation and invariant theory has been used to obtain ML thresholds in two families of multivariate Gaussian models: matrix normal models (see Section~\ref{sec:matrixnormal}) and their higher-order generalization, tensor normal models (see Example~\ref{ex:matrixnormal})~\cite{derksen2020amaximum,derksen2020bmaximum}.}

\subsection*{Organization} To address readers with different backgrounds, we present preliminaries from invariant theory and statistics in Section~\ref{sec:prelim}. We consider the general setting of a Gaussian group model in Section~\ref{sec:generalG}. 
We then study matrix normal models in Section~\ref{sec:matrixnormal}, followed by transitive directed acyclic graphs, in Section~\ref{sec:tdag}.

\section{Preliminaries} \label{sec:prelim}

\subsection{Maximum likelihood estimation}
\label{sec:MLEprelim}
A statistical model is a set of probability distributions. In this paper we consider multivariate Gaussian distributions with mean zero.
The density function of an $m$-dimensional Gaussian 
with mean zero and covariance matrix $\Sigma$~is 
\begin{align*}
    f_\Sigma(y) = \frac{1}{\sqrt{\det( 2 \pi \Sigma)}}
    \exp \left( -\frac{1}{2} y\T \Sigma^{-1} y \right),
\end{align*}
where $y \in \RR^m$ and $\Sigma$ is in the cone of $m \times m$ positive definite matrices, which we denote by $\PD_m$. 
We often consider the concentration matrix $\Psi = \Sigma^{-1}$. A Gaussian model is determined by a set of concentration matrices, i.e. a subset of $\PD_m$.

A maximum likelihood estimate (MLE) is a point in the model that maximizes the likelihood of observing some data $y = (y_1, \dots, y_n)$, where $n$ is the sample size. That is, an MLE maximizes the \emph{likelihood function}
\begin{equation}\label{eq:likelihood}
   L_y(\theta) = f_\theta(y_1) \cdots f_\theta(y_n), 
\end{equation}
where the model is parametrized by $\theta \in \Theta$. It is often convenient to work with the \emph{log-likelihood function} $\ell_y = \log L_y$, which has the same maximizers.

For Gaussian models $\mathcal{M} \subseteq \PD_m$, the data is a tuple $Y = (Y_1, \ldots, Y_n) \in (\RR^m)^n$. The likelihood function \eqref{eq:likelihood} is
\[L_Y(\Psi) = \prod_{i=1}^n f_{\Psi^{-1}}(Y_i).\]
The log-likelihood function can be written, up to additive and multiplicative constants, as
\begin{equation}\label{eqn:gaussianlikelihood}
    \ell_{Y} (\Psi) = \log \det (\Psi) - \mathrm{tr} (\Psi S_Y),
\end{equation}
where $S_Y = \frac{1}{n} \sum_{i=1}^n Y_i Y_i\T$ is the sample covariance matrix, an $m\times m$ positive semi-definite matrix. It is well-known that the unique maximizer of the likelihood over the positive definite cone is $\hat{\Psi} = S_Y^{-1}$, if $S_Y$ is invertible. If $S_Y$ is not invertible, the likelihood function is unbounded and the MLE does not exist, see \cite[Proposition~ 5.3.7]{Sullivant}. 

The minimum number of samples needed for an MLE to generically exist is the \emph{maximum likelihood threshold} ($\mlt$) of a model. The minimum number of samples needed for the likelihood to be generically bounded is denoted by $\mltb$. By \emph{generically}, we mean that a property holds away from an algebraic hypersurface. Hence, it will hold almost surely, i.e. outside of a set of Lebesgue measure zero. As an example, the discussion above says that $\mlt = \mltb = m$ when the Gaussian model is the full positive definite cone, $\mathcal{M} = \PD_m$.

\subsection{Invariant theory}
\label{subsec:InvariantTheory}

This section gives a friendly guide to our invariant theory setting, following~\cite{Wallach}.
We explain how our seemingly special setting fits into  usual terminology  of invariant theory in Remark~\ref{rem:invTheory}.

Invariant theory studies actions of a group $G$ and notions of stability with respect to this action. In this article we work with linear actions on a real or complex vector space. Such a linear action corresponds to a representation $\varrho \colon G \to \GL_m(\KK)$, i.e. each group element $g \in G$ is assigned an invertible matrix in $\GL_m(\KK)$ where $\KK$ is $\RR$ or $\CC$. The group element $g \in G$ acts on $\KK^m$ by left multiplication with the matrix $\varrho(g)$. For a vector $v \in \KK^m$, we define the capacity to be $\capac(v) := \inf_{g \in G} \| g \cdot v \|^2$.
Here and throughout the paper $\Vert \cdot \Vert$ denotes the Euclidean norm for vectors and Frobenius norm for matrices.
We now define the four notions of stability for such an action.

\begin{definition}
\label{def:StabilityNotions}
Let $v \in \KK^m$. We denote the orbit of $v$ by $G \cdot v$, the orbit closure with respect to the Euclidean topology by $\overline{G \cdot v}$ and the stabilizer of $v$ by $G_v$. We say $v$ is
    \begin{itemize}\itemsep 3pt
        \item[(a)] \emph{unstable}, if $0 \in \overline{G \cdot v}$, i.e. $\capac(v) = 0$.
        \item[(b)] \emph{semistable}, if $0 \notin \overline{G \cdot v}$, i.e. $\capac(v) > 0$.
        \item[(c)] \emph{polystable}, if $v \neq 0$ and $G \cdot v$ is closed.
        \item[(d)] \emph{stable}, if $v$ is polystable and $G_v$ is finite.
    \end{itemize}
The set of unstable points is called the \emph{null cone} of the group action.
\end{definition}

The orbit and orbit closure of $v$ only depend on the group $\varrho(G)$. Thus, when studying the notions from Definition~\ref{def:StabilityNotions}(a)--(c) we can assume $G \subseteq \GL_m$ after restricting to the image of $\varrho$. 
We call $G \subseteq \GL_m$ \emph{Zariski closed} if $G$ is the zero locus of a set of polynomials in the matrix entries. 
The transpose of $g \in G$ is denoted by $g\T$ and the Hermitian transpose by $g^\ast$.
We say that a group $G$ is \emph{self-adjoint} if $g \in G$ implies $g\T \in G$ (for $\KK = \RR$), or if $g \in G$ implies $g^\ast \in G$ (for $\KK = \CC$).

Next, we introduce the moment map and state the Kempf-Ness theorem, a crucial ingredient for many of our results.
We consider $G \subseteq \GL_m(\KK)$, a Zariski closed and self-adjoint subgroup.
For each vector $v \in \KK^m$, we study the map
\begin{equation*}
    \gamma_v \colon G \longrightarrow \RR, \quad
    g \longmapsto \Vert g  v \Vert^2,
\end{equation*}
and note that the infimum of $\gamma_v$ is the capacity of $v$.
Since $G$ is defined by polynomial equations, we can consider its tangent space $T_{I_m} G \subseteq \KK^{m \times m}$ at the identity matrix $I_m$, 
and we can compute the differential of the map $\gamma_v$ at the identity:
\begin{equation*}
    D_{I_m} \gamma_v \colon 
    T_{I_m} G
    \longrightarrow \RR, \quad
    \dot g \longmapsto  2 \, \mathrm{Re} [ \tr (\dot{g} v v^{\ast}) ].
\end{equation*}
The \emph{moment map} $\mu$ assigns this differential to each vector $v$, i.e.
\begin{equation*}
    \mu \colon \KK^m \longrightarrow \mathrm{Hom}_{\RR}( T_{I_m} G, \RR ), \quad
    v \longmapsto D_{I_m} \gamma_v.
\end{equation*}
The moment map vanishes at a vector $v$ if and only if the identity matrix $I_m$ is a critical point of the map $\gamma_v$.
Now we are ready to formulate the Kempf-Ness theorem, which is due to \cite{KempfNess} for $\KK=\CC$. The first proof for $\KK = \RR$ was given in \cite{RichardsonSlodowy}.

\begin{theorem}[Kempf-Ness]
\label{thm:KempfNess}
Let $G \subseteq \GL_m(\KK)$ be a Zariski closed self-adjoint subgroup with moment map $\mu$, where $\KK \in \lbrace \RR, \CC \rbrace$.
If $\KK=\RR$, let $K$ be the set of orthogonal matrices in $G$. 
If $\KK = \CC$, let $K$ be the set of unitary matrices in $G$.
For $v \in \KK^m$, we have:
    \begin{itemize} \itemsep 3pt
        \item[(a)] The vector $v$ is of minimal norm in its orbit if and only if $\mu(v)=0$.
        
        \item[(b)] If $\mu(v)=0$ and $w \in G \cdot v$ is such that $\|v\| = \| w \|$, then $w \in K \cdot v$.
        
        \item[(c)] If the orbit $G \cdot v$ is closed, then there exists some $w \in G \cdot v$ with $\mu(w)=0$.
        
        \item[(d)] If $\mu(v)=0$, then the orbit $G \cdot v$ is closed.
        
        \item[(e)] The vector $v$ is polystable if and only if there exists $0 \neq w \in G \cdot v$ with  $\mu(w)=0$.
        
        \item[(f)] The vector $v$ is semistable if and only if there exists $0 \neq w \in \overline{G \cdot v}$ with $\mu(w)=0$.
    \end{itemize}
\end{theorem}

\begin{proof}
Parts (a)--(d) are \cite[Theorems~3.26 and 3.28]{Wallach} while part~(e) is a direct consequence of (c) and (d). Part~(f) follows from the fact that any orbit closure $\overline{G \cdot v}$ contains a unique closed orbit, which is not the zero orbit if and only if $v$ is semistable. For $\KK = \CC$ this can be found in \cite[Theorem~3.20]{Wallach} and for $\KK = \RR$ we refer to \cite[Section~9.3]{RichardsonSlodowy} or \cite[Theorem~1.1(iii)]{RealGIT}. For the latter, note that \cite[Condition~(1)]{RealGIT} is satisfied in our setting by \cite[Theorem~2.16]{Wallach}.

The assumption that $G$ is connected, which appears in \cite[Theorem~3.26]{Wallach}, is not needed here, by the following argument.
If $G^\circ$ is the identity component of $G$, then the quotient group $G/G^\circ$ is finite and its elements can be represented by unitary matrices, by the polar decomposition \cite[Theorem~2.16]{Wallach}.
Hence (a)-(f) above depend only on $G^\circ$.
\end{proof}

The following result relates the capacity over $\CC$ to the capacity over $\RR$.

\begin{proposition}\label{prop:realVsComplexCapacity}
Let $G_{\RR}$ be a Zariski closed self-adjoint subgroup of $\GL_m(\RR)$ and denote by $G_{\CC}$ its Zariski closure in $\GL_m(\CC)$. Let $\capac_{\KK}(v)$ be the capacity of $v \in \KK^m$ under $G_{\KK}$ and denote the null cone under left multiplication with $G_{\KK}$ by $\mathcal{N}_{\KK}$. Then, for $v \in \RR^m$, we have the equality of capacities $\capac_{\RR}(v) = \capac_{\CC}(v)$.
In particular, $\mathcal{N}_{\RR} = \mathcal{N}_{\CC} \cap \RR^m$.
\end{proposition}

\begin{proof}
The group $G_{\CC} \subseteq \GL_m(\CC)$ is self-adjoint by \cite[Lemma~3.29]{Wallach}. The capacity 
$\capac_{\KK}(v)$ is attained at all elements of minimal norm in the closed orbit contained in $\overline{G_{\KK} \cdot v}$, by Kempf-Ness.
Hence we can reduce to studying a closed orbit $G_{\RR} \cdot v$.
If $w$ is of minimal norm in $G_{\RR} \cdot v$, then it is of minimal norm in $G_{\CC} \cdot v$ by \cite[Lemma~3.31]{Wallach} or \cite[Lemma~8.1]{RichardsonSlodowy}. Thus, $G_{\CC} \cdot w$ is closed by Kempf-Ness and hence $\|w\|^2 = \capac_{\CC}(v)$.
\end{proof}

\begin{remark}
\label{rem:invTheory}
We relate our special setting to the usual setting from invariant theory, where one considers a linearly reductive group $G$ over $\KK \in \lbrace \RR, \CC \rbrace$.
For such a group, any finite dimensional rational representation $\varrho \colon G \to \GL(V)$ over $\KK$ on a vector space $V$  is semisimple (also called fully reducible),
i.e. the representation decomposes into irreducible representations. 
Moreover, $\varrho(G) \subseteq \GL(V)$ is a closed algebraic subgroup, see e.g. \cite[Theorem~5.39]{MilneAlgGroups}.
Hence, there exists an inner product on $V$ such that $\varrho(G) \subseteq \GL(V)$ is self-adjoint, see \cite[Theorem~7.1]{Mostow} for $\KK = \CC$ and \cite[Theorem~7.2]{Mostow} for $\KK = \RR$.
\end{remark}

\section{Gaussian Group models}
\label{sec:generalG}

We construct Gaussian models from representations $G \to \GL(V)$ of a group $G$ on a real vector space $V$.
This extends the idea that log-linear models are orbits of the action by a torus, which is utilized in~\cite{companion}.
Our construction only depends on the image of the group $G$ inside $\GL(V)$.
We view each group element as an $m \times m$ invertible matrix by fixing an isomorphism $V \cong \RR^m$.
The \emph{Gaussian group model} given by $G$
is the multivariate Gaussian model consisting of all distributions of mean zero whose concentration matrices lie in the set
\[ \mathcal{M}_G = \{ g\T g \mid g \in G \}.\]
 Equivalently, we take $\mathcal{M}_G$ to be the model consisting of distributions whose covariance matrices are of the form $g g\T$. This is an instance of a \emph{transformation family}, a statistical model on which a group acts transitively, see \cite{exptransfmod}. Our construction includes familiar examples of statistical models.

\begin{example}
\label{ex:standardGaussian}
When $G$ is the general linear group $\GL(V)$, every concentration matrix lies in  $\mathcal{M}_G$ and we get a standard multivariate Gaussian of dimension $\dim(V)$, see Section~\ref{sec:MLEprelim}.
\hfill\exSymbol
\end{example}

\begin{example}
\label{ex:diagonalGaussian}
When $G$ is the torus of diagonal matrices $\GT(V)$, the concentration matrices $g\T g$ are also diagonal and the Gaussian group model consists of $\dim(V)$ independent univariate Gaussian variables.
\hfill\exSymbol
\end{example}

Given two matrices $A_k \in \RR^{m_k \times m_k}$, 
the Kronecker product $A_1 \otimes A_2$ is a an $m_1 m_2 \times m_1 m_2$ matrix. Its rows are indexed by $(i_1,i_2)$, and its columns by $(j_1, j_2)$, where the indices $i_k$ and $j_k$ range from $1$ to $m_k$. The entry of $A_1 \otimes A_2$ at index $((i_1,i_2),(j_1,j_2))$ is $(A_1)_{i_1 j_1} (A_2)_{i_2 j_2}$. 

\begin{example}\label{ex:matrixnormal}
Consider the subset of $\GL_{m_1 m_2}$ given by the image of
\begin{align*}
\GL_{m_1} \times \GL_{m_2} &\longrightarrow \GL_{m_1m_2} \\
(g_1, g_2) &\longmapsto g_1 \otimes g_2.
\end{align*}
The concentration matrices in the Gaussian group model are those of the form
\[ (g_1 \otimes g_2)\T (g_1 \otimes g_2) = g_1\T g_1 \otimes g_2\T g_2, \]
a Kronecker product of an $m_1 \times m_1$ concentration matrix and an $m_2 \times m_2$ concentration matrix.
These Gaussian group models are known as matrix normal models, which we discuss in detail in Section~\ref{sec:matrixnormal}.
This setting can be extended to tensor normal models  under the map
 $(g_1, \ldots, g_d) \mapsto g_1 \otimes \cdots \otimes g_d$.
\hfill\exSymbol
\end{example}

We discuss further examples in the context of directed graphical models in Section~\ref{sec:tdag}.
Now, we describe maximum likelihood estimation for the Gaussian group model  given by $G$. 

The log-likelihood function is
\begin{equation}
\label{eq:generalLogLikelihood}
    \ell_Y(\Psi) = \log \det (\Psi) - \tr(\Psi S_Y),
\end{equation}
where $S_Y = \frac{1}{n} \sum_{i=1}^n Y_i Y_i\T$ is the sample covariance matrix, see~\eqref{eqn:gaussianlikelihood}. An MLE is a concentration matrix in $\mathcal{M}_G$ that maximizes the log-likelihood.

Next, we describe how finding the MLE relates to finding the capacity of the tuple~$Y$.
A consequence of our results is that algorithms to find the capacity can be used to find the MLE in Gaussian group models. 
For example, we can apply methods described in~\cite{PeterAvi3} to the settings of Theorems~\ref{thm:MLEversusStabilityReal} and~\ref{thm:MLEversusStabilityComplex}.

\subsection{Equivalence of optimization problems}
\label{subsec:equivalenceOptimization}

We compare the maximization of the log-likelihood to the minimization of the norm $\| g \cdot Y \|^2$ where $Y$ is a tuple of samples and $g$ is an element of the group. The action of the group $G$ on the tuple $Y = (Y_1, \ldots, Y_n) \in V^n$ is given by $ g \cdot Y
= (g  Y_1, \ldots, g  Y_n)$, i.e. when considering the action on $V^n$ the group $G$ is diagonally embedded in $\GL(V^n)$.
We can rewrite the norm as 
\begin{equation}
\label{eqn:Gnormtrace}
    \Vert g \cdot Y \Vert^2
    = \sum_{i=1}^n (g Y_i)\T g Y_i
    = n \, \tr (g\T g S_Y).
\end{equation}

We compare this expression for the norm with the log-likelihood in~\eqref{eq:generalLogLikelihood}. The term appearing with $S_Y$ in the trace is $\Psi$ in the log-likelihood and $g\T g$ in the norm. This explains our choice to let the Gaussian group model  consist of distributions with concentration matrix $g\T g \in \mathcal{M}_G$.

Combining the expressions for the norm and the log-likelihood, we see that maximizing the log-likelihood over concentration matrices in the model $\mathcal{M}_G$ is equivalent to minimizing 
\begin{equation*}
    -\ell_Y (g\T g) = 
    \frac{1}{n} \, \Vert g \cdot Y \Vert^2 - \log \det (g\T g)
\end{equation*}
over $g \in G$.
We show that this minimization can be done in two steps.
First, we minimize the norm over the subgroup $\GSL^{\pm}$, consisting of matrices in $G$ of determinant $\pm 1$. Then, we find the scalar multiple of this matrix that minimizes the overall expression. For this, we require that the group $G$ is closed under non-zero scalar multiples.

\begin{proposition} 
\label{prop:twoStepGeneral}
Let $Y \in V^n$ be a tuple of samples.
If the group $G \subseteq \GL(V)$ is closed under non-zero scalar multiples, 
the supremum of the log-likelihood~\eqref{eq:generalLogLikelihood} over $\mathcal{M}_G$ is the double infimum
\begin{equation*}
-  \inf_{\lambda \in \RR_{>0}}
\left(
\frac{\lambda}{n}
\left(
\inf_{h \in \GSL^{\pm}} \Vert h \cdot Y \Vert^2
\right)
-  \dim(V) \log \lambda
\right).
\end{equation*}
\\
The MLEs, if they exist, are the matrices $\lambda h\T h$, where $h$ minimizes $\| h \cdot Y\|$
under the action of $\GSL^{\pm}$ on $V^n$, and $\lambda \in \RR_{>0}$ is the unique value minimizing the outer infimum.
\end{proposition}

\begin{proof}
Maximizing $\ell_Y(\Psi)$ over $\mathcal{M}_G$ is equivalent to minimizing
\begin{align*}
    f \colon G &\to \RR \\ g &\mapsto
    \frac{1}{n} \| g \cdot Y \|^2 - \log \det (g\T g),
\end{align*}
since $f(g)$ only depends on the positive definite matrix $g\T g$.
We  write $g \in G$ as $g = \tau h$, where $\tau \in \RR_{> 0}$ and $h \in \GSL^{\pm}$.
Using $g\T g = \tau^2 h\T h$, and setting $\lambda := \tau^2$ and $m := \dim(V)$, we have
    \begin{align*}
        f(g) = \frac{\tau^2}{n} \| h \cdot Y \|^2 - \log \det(\tau^2 h\T h) 
        = \frac{\tau^2}{n} \| h \cdot Y \|^2 -  \log (\tau^{2 m}) 
        = \frac{\lambda}{n} \| h \cdot Y \|^2 - m \log (\lambda).
        \end{align*}
The minimum value of the function $ \lambda \mapsto \lambda C - \log(\lambda)$
is $\log(C)+1$ for $C > 0$, which increases as $C$ increases. 
Hence, to minimize $f$, we can first find the minimal norm in the orbit closure and then minimize the univariate function in $\lambda$, 
i.e.
    \begin{align*}
        \inf_{g \in G} f(g) 
        = \inf_{\lambda \in \RR_{>0} } \left( \frac{\lambda}{n} \left( \inf_{h \in \GSL^{\pm}} \| h \cdot Y \|^2 \right) - m \log \lambda \right).
    \end{align*}
Furthermore, an MLE is a matrix $\hat{\Psi}  \in \mathcal{M}_G$ that maximizes $\ell_Y(\Psi)$. Comparing $\ell_Y(\Psi)$ with the infimum in the claim, we see that the MLEs are all matrices $\hat{\Psi} = g\T g = \lambda h\T h$,
where $g = \sqrt{\lambda} h$, and $h$ and $\lambda$ minimize the inner and outer infima respectively.
\end{proof}

The group $\GSL^{\pm}$ may split into two parts:
$\GSL^+$ consisting of matrices in $G$ of determinant one, 
and $\GSL^-$ consisting of matrices of determinant $-1$. 
If we prefer to optimize over one part, say $\GSL^+$, we can compute the capacity of $Y$ under $\GSL^{\pm}$ by doing two minimizations.
A fixed matrix $h' \in \GSL^-$ gives a bijection between $\GSL^+$ and $\GSL^-$
via $h \mapsto h h'$. Hence
we can minimize $\Vert h \cdot Y \Vert$ over $\GSL^\pm$ by minimizing both $\Vert h \cdot Y \Vert$ and $\Vert h \cdot (h' \cdot Y) \Vert$ over $\GSL^+$. However, we can ignore neither $\GSL^+$ nor $\GSL^-$. The following is an example of a group, closed under non-zero scalar multiples, such that the norm $\Vert h \cdot Y \Vert$ can be attained at one but not the other.

\begin{example}
\label{ex:minimumOnNegativeComponent}
Let the group $G$ consist of non-zero scalar multiples of block-diagonal $6 \times 6$ matrices of the form
\begin{equation}
    \label{eqn:oxfordgroup}
    \begin{bmatrix} M & 0 & 0 \\ 0 & S_1 M S_1^{-1} & 0 \\ 0 & 0 & S_2 M S_2^{-1} \end{bmatrix} , \quad \text{where} \quad S_1 = \begin{bmatrix} 1 & 2 \\ 2 & 1 \end{bmatrix}, \quad
S_2 = \begin{bmatrix} 1 & 0 \\ 0 & 2 \end{bmatrix} ,
\end{equation} 
and $M \in O_2$ is an orthogonal $2 \times 2$ matrix. The component $\GSL^+$ consists of matrices in~\eqref{eqn:oxfordgroup} where $M$ is special orthogonal, while the component $\GSL^-$ consists of matrices in~\eqref{eqn:oxfordgroup} where $M$ is orthogonal with determinant $-1$. Note that although the group $G$ contains matrices of determinant $-1$, it does not contain any orthogonal matrices of determinant $-1$. 

The norm of $\Vert g \cdot Y \|$, for a tuple of samples $Y$, can be expressed in terms of the sample covariance matrix $S_Y$.
Consider the tuple of four samples given by
\[ Y = \begin{bmatrix} 0 & 0 & 0 & 0 \\ 0 & 0 & 0 & 0 \\ 2 & 0 & 0 & 0 \\ 0 & 2 \sqrt{2} & 0 & 0 \\ 0 & 0 & 0 & 2 \sqrt{5} \\ 0 & 0 & \frac{6 \sqrt{5}}{5} & \frac{8 \sqrt{5}}{5} \end{bmatrix}, \quad \text{with} \quad S_Y = \begin{bmatrix} 0 & 0 & 0 \\ 0 & S_2 & 0 \\ 0 & 0 & S_1^2 \end{bmatrix} .\]
The capacity problem can be rewritten as minimizing the trace $\tr(g\T g S_Y)$ over matrices $g \in \GSL^\pm$, by~\eqref{eqn:Gnormtrace}, to give
\[  \inf_{h \in \GSL^\pm}  \| h \cdot Y \|^2 = 4 \cdot \inf_{M \in O_2}  \left[ \tr \left( (S_1 M S_1^{-1})\T (S_1 M S_1^{-1}) S_2 \right) +  \tr \left(( S_2 M S_2^{-1})\T ( S_2 M S_2^{-1}) S_1^2 \right) \right].\]

We can parametrize the $2 \times 2$ special orthogonal matrices by $P$ and the $2 \times 2$ orthogonal matrices of determinant $-1$ by $Q$ where
\[ P = \begin{bmatrix} a & b \\ -b & a \end{bmatrix}, \qquad Q = \begin{bmatrix} -a & -b \\ -b & a \end{bmatrix}, \quad \text{with} \quad  a,b \in \RR, \quad \text{and} \quad a^2 + b^2 = 1. \]
 Then the minimization problems over $\GSL^+$ and $\GSL^-$ can be rewritten as
\[ \inf_{h \in \GSL^+} \frac{1}{4} \| h \cdot Y \|^2 = \min_{ a^2+b^2=1} \left( 13 a^2 - \frac{44}{3} ab + \frac{419}{12} b^2 \right), \]
\[ \inf_{h \in \GSL^-} \frac{1}{4} \| h \cdot Y \|^2 = \min_{ a^2+b^2=1} \left( \frac{71}{3} a^2- \frac{28}{3} ab + \frac{97}{4} b^2 \right) .\]
Note that both infima can only be attained for $a$ and $b$ having the same sign, because of the negative coefficients of $ab$; we assume $a,b \geq 0$.
Substituting $b = \sqrt{1-a^2}$ in the latter minimum, we see that
\[ \frac{71}{3} a^2 + \frac{97}{4} (1 - a^2) - \frac{28}{3} a \sqrt{1 - a^2} \geq \frac{97}{4} + \left( \frac{71}{3} - \frac{97}{4} \right) - 
\frac{28}{3} \cdot \frac{1}{2}  = 19.
\]
In contrast, setting $a = 1$ and $b=0$ in the former minimum gives a value of 13. 
Hence $ \inf_{h \in \GSL^+} \| h \cdot Y \|^2 < \inf_{h \in \GSL^-} \| h \cdot Y \|^2$. Multiplying $Y$ by a fixed matrix in $\GSL^-$ gives a tuple of samples where the strict inequality is reversed, and the minimum is attained only at the negative component $\GSL^-$.
\hfill\exSymbol
\end{example}

\subsection{Relating  stability  to  the  MLE}
\label{subsec:StabilityMLE}

We use Proposition~\ref{prop:twoStepGeneral} to prove the following correspondence between stability notions and MLE existence. 

\begin{theorem}
\label{thm:gaussianMLEstable}
Consider a tuple $Y \in V^n$ of samples, and a group $G \subseteq \GL(V)$ that is closed under non-zero scalar multiples.
The stability under the action of $\GSL^{\pm}$ on $V^n$ is related to ML estimation for the Gaussian group model $\mathcal{M}_G$ as follows.
\[ \begin{matrix} 
(a) & Y \text{ unstable} & \Leftrightarrow & \ell_Y \text{ not bounded from above} \\
(b) & Y \text{ semistable} & \Leftrightarrow & \ell_Y \text{ bounded from above} \\ 
(c) & Y \text{ polystable} & \Rightarrow & \text{MLE exists}
\end{matrix} \]
\end{theorem}

\begin{proof}
If $Y$ is unstable then $C := \inf_{h \in \GSL^{\pm}} \| h \cdot Y \|^2 = 0$. Hence the outer infimum from Proposition~\ref{prop:twoStepGeneral} equals $- \infty$, so the supremum of $\ell_Y$ is infinite. Conversely, if $Y$ is semistable, then $C > 0$ and thus the outer infimum from Proposition~\ref{prop:twoStepGeneral} is some real number and $\ell_Y$ is bounded from above. This gives parts (a) and (b).

If $Y$ is polystable, then the infimum $C > 0$ is attained for some $h \in \GSL^{\pm}$ and $\lambda h\T h$ is an MLE, where $\lambda \in \RR_{>0}$ minimizes the outer infimum in Proposition~\ref{prop:twoStepGeneral}.
\end{proof}

\begin{remark}
\label{rem:orthogonalMatricesNegativeDet}
Assume that $G$ contains an orthogonal matrix of determinant $-1$, say $o \in G$.
Then minimizing the norm $\Vert h \cdot Y \Vert$ over $\GSL^{\pm}$ is equivalent to minimizing it over $\GSL^+$.
Hence, in this case, 
Proposition~\ref{prop:twoStepGeneral} 
and Theorem~\ref{thm:gaussianMLEstable}
both hold for $\GSL^{+}$ as well as $\GSL^{\pm}$. 
This is because we can write
$g \in G$ as $g = \tau o h$, where $\tau \in \RR_{>0}$ and $h \in \GSL^+$, and then follow the computations in the proof of Proposition~\ref{prop:twoStepGeneral}.
\end{remark}

If we assume that our group $G$ is Zariski closed and self-adjoint, we can strengthen Theorem~\ref{thm:gaussianMLEstable} using the Kempf-Ness theorem over $\RR$.
These additional assumptions hold for Examples~\ref{ex:standardGaussian}, \ref{ex:diagonalGaussian}, and \ref{ex:matrixnormal}.
On the statistics side, self-adjointness implies that the set of concentration matrices in  $\mathcal{M}_G$ is equal to the set of covariance matrices in the model.

\begin{lemma}
\label{lem:OrthogonalMatrixNegativeDet}
Let $G \subseteq \GL(V)$ be a Zariski closed self-adjoint group, closed under non-zero scalar multiples. If there is an element of $G$ with negative determinant, then $G$ contains an orthogonal matrix of determinant $-1$. In particular, Proposition~\ref{prop:twoStepGeneral} and Theorem~\ref{thm:gaussianMLEstable} still hold after replacing $\GSL^{\pm}$ by $\GSL^+$.
\end{lemma}

\begin{proof}
Pick $g \in G$ with $\det(g) < 0$. Since $G$ is Zariski closed and self-adjoint, the polar decomposition can be carried out in $G$, by~\cite[Theorem~2.16]{Wallach}. 
In particular, there is an orthogonal $o \in G$ and a positive definite $p \in G$ such that $g = op$. Then $\det(g) < 0$ implies $\det(o) < 0$, i.e. $\det(o) = -1$. The second part of the claim follows from Remark~\ref{rem:orthogonalMatricesNegativeDet}.
\end{proof}

As a consequence of Lemma~\ref{lem:OrthogonalMatrixNegativeDet} we work with $\GSL^+$ (instead of $\GSL^{\pm}$) in the following.

\begin{proposition}
\label{prop:MLEsViaStabilizer}
Let $Y \in V^n$ be a tuple of samples, and $G \subseteq \GL(V)$ a Zariski closed self-adjoint group which is closed under non-zero scalar multiples. If $\lambda h\T h$ is an MLE given $Y$, with $h \in \GSL^+$ and $\lambda \in \RR_{>0}$, then all MLEs given $Y$ are of the form $g\T(\lambda h\T h)g$, where $g$ is in the $\GSL^+$-stabilizer of $Y$.
\end{proposition}

\begin{proof}
By Proposition~\ref{prop:twoStepGeneral} for $\GSL^+$, the matrix $h$ minimizes the norm of $Y$ under the action of $\GSL^+$ and hence so does $hg$ for any $g$ in the $\GSL^+$-stabilizer of $Y$.
Therefore, $\lambda (hg)\T hg = g\T(\lambda h\T h)g$ is another MLE. Conversely, by Proposition~\ref{prop:twoStepGeneral} any MLE is of the form $\lambda (h')\T h'$ with $h' \in \GSL^+$ such that
    \begin{equation*}
        \| h' \cdot Y \|^2 = \inf_{\tilde{h} \in \GSL^+} \| \tilde{h} \cdot Y \|^2 = \| h \cdot Y \|^2.
    \end{equation*}
Since $G \subseteq \GL(V)$ is Zariski closed and self-adjoint, $\GSL^+ \subseteq \GL(V)$ is Zariski closed and self-adjoint and so is its diagonal embedding into $\GL(V^n)$. Thus we can apply Kempf-Ness, Theorem~\ref{thm:KempfNess}(b). For the $\GSL^+$ action on $V^n$, there is an orthogonal matrix $o \in \GSL^+$ with $o \cdot (h \cdot Y) = h' \cdot Y$. Hence, $g := h^{-1} \, o^{-1} \, h'$ is in the $\GSL^+$-stabilizer of~$Y$ and using $h' = o h g$ we deduce $\lambda (h')\T h' = g\T (\lambda h\T h)g$.
\end{proof}

With these extra assumptions on the group $G$, we obtain a stronger version of Theorem~\ref{thm:gaussianMLEstable}.
Moreover, with these assumptions we are in the setting of~\cite{PeterAvi3}, 
so we can use their algorithmic methods to compute the capacity in order to find an MLE.
We discuss these connections to algorithms for matrix normal models in Section~\ref{sec:matrixnormal}.

\begin{theorem}
\label{thm:MLEversusStabilityReal}
Let $Y \in V^n$ be a tuple of samples, and 
 $G \subseteq \GL(V)$ a Zariski closed self-adjoint group that is closed under non-zero scalar multiples. 
The stability under the action of $\GSL^+$ on $V^n$ is related to ML estimation for the Gaussian group model $\mathcal{M}_G$ as follows.
\[ \begin{matrix}
(a) & Y \text{ unstable} & \Leftrightarrow & \ell_Y \text{ not bounded from above} & & \\
(b) & Y \text{ semistable} & \Leftrightarrow & \ell_Y \text{  bounded from above} & & \\ 
(c) & Y \text{ polystable} & \Leftrightarrow & \text{MLE exists} & & 
 \\
(d) & Y \text{ stable} & \Rightarrow & \text{finitely many MLEs exist} & \Leftrightarrow &
\text{ unique MLE exists} 
\end{matrix}
\]
\end{theorem}

\begin{proof}
We recall that the action of $\GSL^+$ on $V^n$ is given by the diagonal embedding into $\GL(V^n)$, and that this turns $\GSL^+$ into a Zariski closed self-adjoint subgroup of $\GL(V^n)$ by the assumptions on $G \subseteq \GL(V)$.

By Theorem~\ref{thm:gaussianMLEstable}, it remains to prove the converse implication in (c) and condition~(d). If an MLE given $Y$ exists, then the log-likelihood function $\ell_Y$ is bounded from above and attains its maximum. Hence the double infimum from Proposition~\ref{prop:twoStepGeneral} is attained, and there exists $h \in \GSL^+$ such that $h \cdot Y$ has minimal norm in the orbit of $Y$ under $\GSL^+$. Hence
the orbit is closed by Kempf-Ness, Theorem~\ref{thm:KempfNess}(d), and $Y$ is polystable.

We now prove condition (d).
If $Y$ is stable, its stabilizer $\St_Y$ is finite.
Then there are only finitely many MLEs given $Y$, by Proposition~\ref{prop:MLEsViaStabilizer}.
It remains to show that a tuple $Y$ cannot have finitely many MLEs unless it has a unique MLE.
A tuple $Y$ with finitely many MLEs is polystable, by condition (c).
Moreover, we can relate the stabilizers of $Y$ and $h \cdot Y$ by $\St_{h \cdot Y} = h \, \St_Y \, h^{-1}$.
Combining Propositions~\ref{prop:twoStepGeneral} and~\ref{prop:MLEsViaStabilizer},  we can relate the MLEs given $Y$ to the MLEs given $h \cdot Y$ via 
    \begin{align*}
        \lbrace \text{MLEs given } h \cdot Y \rbrace = \left( h^{-1} \right)\T \lbrace \text{MLEs given } Y \rbrace h^{-1}.
    \end{align*}
Hence, to study the stabilizer and MLE of a polystable $Y$ we can assume that $Y$ is of minimal norm in its orbit under $\GSL^+$.
One of the MLEs given $Y$ is then $\lambda I$, where $\lambda > 0$ minimizes the outer infimum in Proposition~\ref{prop:twoStepGeneral}, and $I$ is the identity matrix of size $\dim(V)$. 

We show that the set $\{ g\T g \mid g \in \St_Y \}$ is either the identity matrix or infinite.
This implies that $Y$ either has a unique MLE or infinitely many MLEs, because the MLEs given $Y$ are the matrices $g\T(\lambda I\T I)g = \lambda g\T g$, where $g \in \St_{Y}$, by Proposition~\ref{prop:MLEsViaStabilizer}.
The group $\St_{Y}$ is self-adjoint by~\cite[Corollary 2.25]{Wallach}.
If it is contained in the set of orthogonal matrices, then $\{ g\T g \mid g \in \St_Y \}$ consists only of the identity matrix.
Otherwise, let $h \in \St_Y$ be non-orthogonal.
Then $h\T \in \St_Y$  and hence $h\T h \in \St_Y$, and this positive definite matrix is not equal to the identity matrix. The matrix $h\T h$ has infinite order, since the eigenvalues of $(h\T h)^N$ are the $N$th powers of the eigenvalues of $h\T h$, and there exist eigenvalues that are not equal to one. 
Since $(h\T h)^N \in \St_Y$ and $((h\T h)^N)\T((h\T h)^N)= (h\T h)^{2N}$,
the set $\{ g\T g \mid g \in \St_Y\}$ is infinite.
\end{proof}

\begin{remark}\label{rem:choiceOfGroup}
In the setting of a Zariski closed self-adjoint group $G$ closed under non-zero scalar multiples, the results in
 Proposition~\ref{prop:twoStepGeneral},
 Proposition~\ref{prop:MLEsViaStabilizer}, and
Theorem~\ref{thm:MLEversusStabilityReal} are unchanged if we replace $\GSL^+$ by
the larger subgroup $\GSL^{\pm}$, by the same argument as in Lemma~\ref{lem:OrthogonalMatrixNegativeDet}.
In fact, we can also replace $\GSL^+$ by the smaller group $\GSL^\circ$, the identity component of $\GSL^+$. This is because the quotient group $\GSL^+ / \GSL^\circ$ is finite and 
every equivalence class has an orthogonal matrix representative, by the polar decomposition~\cite[Theorem~2.16]{Wallach}. The same argument holds for any Zariski-closed self-adjoint subgroup $H$ of $G$ with the same identity component as $\GSL^+$. 
We may not have such choices for groups that are not Zariski closed and self-adjoint, see Example~\ref{ex:minimumOnNegativeComponent}.
\end{remark}

We note that the converse of Theorem~\ref{thm:MLEversusStabilityReal}(d) does not hold by Example~\ref{ex:PolystableNotStableUniqueMLE} from the next section. 
We also stress the importance of the assumption that the group $G$ is self-adjoint for condition (d). This assumption is needed to conclude that the MLE is unique from the fact that there are finitely many MLEs. Indeed, the following example exhibits a Zariski closed group $G$, closed under non-zero scalar multiples, for which there exist samples $Y$ with a finite number of MLEs in the Gaussian group model given by $G$, but not a unique MLE.

\begin{example}\label{ex:nonuniqueMLE}
Let $G$ be generated by $-I$ and all non-zero scalar multiples of a non-orthogonal matrix $M$ with $M^2 = I$. For example, we can take 
\[ M = \begin{bmatrix} \nicefrac12 & 3 \\ \nicefrac14 & \nicefrac{-1}2 \end{bmatrix}. \]
The group consists of non-zero scalar multiples of the matrices $M$ and $I_2$. 
The MLEs to the Gaussian group model $\mathcal{M}_G$ given samples $Y$ are given by group elements $h \in \GSL^\pm$ that minimize the norm $\| h \cdot Y\|$, by Proposition~\ref{prop:twoStepGeneral}. 
Since scaling the matrix by some $\lambda$ scales its determinant by $\lambda^2$, the subset $\GSL^+$ consists of $\pm I_2$, and the subset $\GSL^-$ consists of the matrices $\pm M$. 
Consider the single sample
\[ Y = \begin{bmatrix} 6 \\ 1 \end{bmatrix}. \]
Then 
$ \| M \cdot Y \|^2 = \| Y \|^2, $
and the sample $Y$ has exactly two distinct MLEs. 
\hfill\exSymbol
\end{example}

\subsection{Complex Gaussian models}
\label{sec:complexGaussian}
Invariant theory is more classical over the field of complex numbers than over the reals numbers.
We see in this section that several of our results can be simplified and strengthened when working over $\CC$.
The statistical consequences concern statistical models over the complex numbers, as in~\cite{complexGaussian,goodman1963complex,complexG}. 

We consider a complex vector space $V$ and a subgroup $G \subseteq \GL_{\CC}(V)$ of the complex general linear group on $V$. To view the group elements in $G$ as invertible matrices we fix an isomorphism $V \cong \CC^m$.
The complex \emph{Gaussian group model} $\mathcal{M}_G$ consists of all multivariate distributions of mean zero whose concentration matrix is of the form $g^\ast g$ for some  $g \in G$.
The log-likelihood function becomes
\begin{equation}
\label{eq:logLikelihoodGaussianComplex}
    \ell_Y(\Psi) = \log  \det (\Psi) - \tr(\Psi S_Y), \quad \text{ where }
    S_Y := \frac{1}{n} \sum_{i=1}^n Y_i Y_i^\ast.
\end{equation}
For the action of the group $G$ on a tuple $Y = (Y_1, \ldots, Y_n) \in V^n$ given by $g \cdot Y = (g Y_1, \ldots, g Y_n)$, the norm becomes
\begin{equation*}
    \Vert g \cdot Y \Vert^2 = \sum_{i=1}^n (g Y_i)^\ast gY_i = n \, \tr(g^\ast g S_Y).
\end{equation*}
Hence, as before, maximizing the log-likelihood over concentration matrices in the complex Gaussian group model $\mathcal{M}_G$ is equivalent to minimizing
\begin{equation*}
    - \ell_Y(g^\ast g) = \frac{1}{n} \Vert g \cdot Y \Vert^2  - \log \det (g^\ast g).
\end{equation*}
Analogously to Proposition~\ref{prop:twoStepGeneral}, this can be done in two steps. 
Since we now work over~$\CC$, we only need to compute the capacity under the subgroup $\GSL^+ \subseteq G$ of matrices with determinant one, instead of using $\GSL^{\pm}$. In particular, the situation described in Example~\ref{ex:minimumOnNegativeComponent} cannot happen over $\CC$, and we do not need to consider the extra assumptions in  Remark~\ref{rem:orthogonalMatricesNegativeDet}.

\begin{proposition}
\label{prop:twoStepGeneralComplex}
Let $Y \in V^n$ be a tuple of complex samples.
If the group $G$ is closed under non-zero complex scalar multiples, 
the supremum of the log-likelihood~\eqref{eq:logLikelihoodGaussianComplex} over the model $\mathcal{M}_G$ is the double infimum
\begin{equation*}
-  \inf_{\lambda \in \RR_{>0} }
\left(
\frac{\lambda}{n}
\left(
\inf_{h \in \GSL^{+}} \Vert h \cdot Y \Vert^2
\right)
-  \dim(V) \log \lambda
\right).
\end{equation*}
\\
The MLEs, if they exist, are the matrices $\lambda h^\ast h$, where $h$ minimizes $\| h \cdot Y\|$
under the action of $\GSL^{+}$ on $V^n$, and $\lambda \in \RR_{>0}$ is the unique value minimizing the outer infimum.
\end{proposition}

\begin{proof}
The proof is analogous to the proof of Proposition~\ref{prop:twoStepGeneral}.
The only difference is that we can write $g \in G$
as $g = \tau h$, where $\tau \in \CC \setminus \lbrace 0 \rbrace$ and $h \in \GSL^+$.
Then we see that
\begin{equation*}
    - \ell_Y(g^\ast g) = \frac{|\tau|^2}{n} \Vert h \cdot Y \Vert^2 - \dim(V) \log (|\tau|^2).
\end{equation*}
Setting $\lambda = |\tau|^2$ and continuing as in the proof of Proposition~\ref{prop:twoStepGeneral}, shows the claim.
\end{proof}

Using the same assumptions as in Proposition~\ref{prop:twoStepGeneralComplex}, we see that Theorem~\ref{thm:gaussianMLEstable} holds over $\CC$ after replacing $\GSL^{\pm}$ by $\GSL^+$.
The most important difference between the real and the complex setting is that Theorem~\ref{thm:MLEversusStabilityReal}(d) is an equivalence over $\CC$.
In Example~\ref{ex:PolystableNotStableUniqueMLE}, we will see that this is not true over $\RR$.
In the remainder of this section, we prove this equivalence for complex Gaussian group models given by self-adjoint groups $G$.
We first give an analogue of Proposition~\ref{prop:MLEsViaStabilizer} over~$\CC$.

\begin{proposition}
\label{prop:MLEsViaStabilizerComplex}
Let $Y \in V^n$ be a tuple of complex samples, and $G \subseteq \GL_{\CC}(V)$ be a Zariski closed self-adjoint group, which is closed under non-zero complex scalar multiples. If $\lambda h^\ast h$ is an MLE given $Y$, with $h \in \GSL^+$ and $\lambda \in \RR_{>0}$, then all MLEs given $Y$ are of the form $g^\ast(\lambda h^\ast h)g$, where $g$ is in the $\GSL^+$-stabilizer of $Y$.
\end{proposition}

\begin{proof}
This is proven analogously as Proposition~\ref{prop:MLEsViaStabilizer} using the complex version of Kempf-Ness Theorem~\ref{thm:KempfNess} and Proposition~\ref{prop:twoStepGeneralComplex} instead of Proposition~\ref{prop:twoStepGeneral}.
\end{proof}

\begin{theorem}\label{thm:MLEversusStabilityComplex}
Consider a tuple $Y \in V^n$ of complex samples, and let
 $G \subseteq \GL_{\CC}(V)$ be a Zariski closed self-adjoint group, which is closed under non-zero complex scalar multiples.
The stability under the action of $\GSL^+$ on $V^n$ is related to ML estimation for the complex Gaussian group model $\mathcal{M}_G$ as follows.
\[ \begin{matrix}
(a) & Y \text{ unstable} & \Leftrightarrow & \ell_Y \text{ not bounded from above} & \\
(b) & Y \text{ semistable} & \Leftrightarrow & \ell_Y \text{  bounded from above} & \\ 
(c) & Y \text{ polystable} & \Leftrightarrow & \text{MLE exists}
 & \\
(d) & Y \text{ stable} & \Leftrightarrow &
\text{finitely many MLEs exist}
& \Leftrightarrow 
\text{ unique MLE exists} 
\end{matrix}
\]
\end{theorem}

\begin{proof}
We prove that uniqueness of the MLE given $Y$ implies that $Y$ is stable.
The proofs of the other parts of the theorem are the same as in the real setting in Theorems~\ref{thm:gaussianMLEstable} and~\ref{thm:MLEversusStabilityReal}. 

Let us assume that the MLE given $Y$ exists uniquely. 
We see from (c) that $Y$ is polystable.
Hence, we need to show that the $\GSL^+$-stabilizer of $Y$, denoted by $\St_Y$, is finite. 
For $h \in \GSL^+$ we have $\St_{h \cdot Y} = h \, \St_Y \, h^{-1}$ and, from Proposition~\ref{prop:MLEsViaStabilizerComplex}, we have
    \begin{equation*}
        \lbrace \text{MLEs given } h \cdot Y \rbrace = \left( h^{-1} \right)^\ast \lbrace \text{MLEs given } Y \rbrace h^{-1}.
    \end{equation*}
As in the real setting, this allows us to assume that $Y$ is of minimal norm in its orbit under $\GSL^+$.
Then $\lambda I$ is the MLE given $Y$, where $\lambda \in \RR_{>0}$ minimizes the outer infimum in Proposition~\ref{prop:twoStepGeneralComplex}.
Since the matrix $\lambda I$ is the unique MLE, the stabilizer $\St_Y$ is contained in the group  of unitary matrices in $G$, by Proposition~\ref{prop:MLEsViaStabilizerComplex}.
In particular, $\St_Y$ is $\CC$-compact.
 As the subgroup $\St_Y$ is also Zariski closed (defined by the equations $gY = Y$)
we conclude that $\St_Y$ is finite.
\end{proof}

\section{Matrix Normal Models}
\label{sec:matrixnormal}
In this section we study matrix normal models, which we have already seen in Example~\ref{ex:matrixnormal}.
Consider the multivariate Gaussian of dimension $m = m_1 m_2$. 
A matrix normal model is a sub-model consisting of covariance matrices that factor as a Kronecker product $\Sigma_1 \otimes \Sigma_2$ where $\Sigma_i \in \PD_{m_i}$. 
Setting $\Psi_1 := \Sigma_1^{-1}$ and $\Psi_2 := \Sigma_2^{-1}$,
we can write the log-likelihood function \eqref{eqn:gaussianlikelihood} for the matrix normal model as
\begin{equation}
\label{eqn:matrixnormallikelihood}
    \ell_Y(\Psi_1, \Psi_2)
    =  \, m_2 \, \log \det (\Psi_1)
     +  \, m_1 \, \log \det (\Psi_2)
     - \frac{1}{n} \mathrm{tr} \left( \Psi_1 \sum_{i=1}^n Y_i \Psi_2 Y_i\T \right).
\end{equation}
An MLE is a concentration matrix
$\hat{\Psi}_1 
\otimes \hat{\Psi}_2 \in \PD_{m_1} \otimes \PD_{m_2}$ that maximizes the log-likelihood. Unless specified, we refer to matrix normal models over the real numbers and abbreviate $\GL_m(\RR)$ and $\SL_m(\RR)$ to $\GL_m$ and $\SL_m$ respectively.

\subsection{Relating norm minimization to ML estimation}
We describe how to specialize our results for Gaussian group models from Section~\ref{sec:generalG} to matrix normal models.
For this, consider the left-right action of $\GL_{m_1} \times \GL_{m_2}$ on $(\RR^{m_1 \times m_2})^n$ given by
\begin{equation}
\label{eq:SLaction}
    g \cdot Y := (g_1 Y_1 g_2\T, \ldots, g_1 Y_n g_2\T), 
\end{equation}
where $Y = (Y_1, \ldots, Y_n)$ is a sample tuple in $(\RR^{m_1 \times m_2})^n$ and $g = (g_1, g_2) \in \GL_{m_1} \times \GL_{m_2}$.
The left-right action induces the representation
\[\varrho \colon \GL_{m_1} \times \GL_{m_2} \to \GL_{m_1 m_2}, \quad (g_1,g_2) \mapsto g_1 \otimes g_2
\]
and the matrix normal model arises as the Gaussian group model of $G \!:=\! \varrho(\GL_{m_1} \!\!\times\! \GL_{m_2})$.

The subgroup $G \subseteq \GL_{m_1 m_2}$ is Zariski closed, self-adjoint and closed under non-zero scalar multiples.
Therefore, our results from the previous section apply to the action of~$\GSL^+$.  However, it is possible and more convenient to directly work with the left-right action of $\SL_{m_1} \times \SL_{m_2}$.
The following theorem makes this precise.

\begin{theorem}
\label{thm:bigTheoremMatrixNormal}
Let $Y \in (\RR^{m_1 \times m_2})^n$ be a matrix tuple.
The supremum of the log-likelihood $\ell_Y$ in~\eqref{eqn:matrixnormallikelihood} over $\PD_{m_1} \times \PD_{m_2}$ is given by the double infimum
\begin{equation}
    \label{eqn:doubleinf}
    -\inf_{\lambda \in \RR_{>0} } \left( \frac{\lambda}{n} \left( \inf_{h \in \SL_{m_1} \times \SL_{m_2}} \| h \cdot Y \|^2 \right) -  m_1 m_2 \log \lambda \right) .
\end{equation}  
The MLEs, if they exist, are the matrices of the form $\lambda h_1\T h_1 \otimes h_2\T h_2$, where $h = (h_1,h_2)$ minimizes $\| h \cdot Y \|$ under the left-right action of $\SL_{m_1} \times \SL_{m_2}$, and $\lambda \in \RR_{>0}$ is the unique value that minimizes the outer infimum. 

If there are several MLEs given $Y$, they are related via the stabilizer of $Y\!$ in $\SL_{m_1} \!\!\times\! \SL_{m_2}$.
More precisely, 
every $(g_1,g_2)$ in the stabilizer of $Y$
yields an MLE $\lambda g_1\T h_1\T h_1 g_1 \otimes g_2\T h_2\T h_2 g_2$ and, 
conversely, every MLE given $Y$ is of this form.

The stability under the left-right action of $\SL_{m_1} \times \SL_{m_2}$ is related to ML estimation via:
\[ \begin{matrix} 
(a) & Y \text{ unstable} & \Leftrightarrow & \ell_Y \text{ not bounded from above} \\
(b) & Y \text{ semistable} & \Leftrightarrow & \ell_Y \text{ bounded from above} \\ 
(c) & Y \text{ polystable} & \Leftrightarrow & \text{MLE exists}
 \\
(d) & Y \text{ stable} & \Rightarrow & \text{MLE exists  uniquely}
\end{matrix}
\]
\end{theorem}

\begin{proof}
The subgroup $H := \varrho(\SL_{m_1} \times \SL_{m_2}) \subseteq G$ is Zariski closed, self-adjoint and shares the same identity component as $\GSL^+$. Thus Propositions~\ref{prop:twoStepGeneral}, \ref{prop:MLEsViaStabilizer} and Theorem~\ref{thm:MLEversusStabilityReal} apply to $H$ as well, by Remark~\ref{rem:choiceOfGroup}.
Furthermore, the kernel of $\varrho$ when restricted to $\SL_{m_1} \times \SL_{m_2}$ is finite. Hence, the stability notions in Definition~\ref{def:StabilityNotions}(a)--(d) coincide for $\SL_{m_1} \times \SL_{m_2}$ and $H$, so we can consider $\SL_{m_1} \times \SL_{m_2}$ instead of its image $H$ under $\varrho$.
\end{proof}

We have seen in Theorem~\ref{thm:MLEversusStabilityComplex} that over the complex numbers, the converse of Theorem~\ref{thm:bigTheoremMatrixNormal}(d) also holds. However, over the reals there exist matrix tuples $Y$ with a unique MLE but an infinite stabilizer, as the following example shows.

\begin{example}
\label{ex:PolystableNotStableUniqueMLE}
We set $m_1 = m_2 = n = 2$ and take $Y \in (\RR^{2 \times 2})^2$, where
    \begin{equation*}
        Y_1 = \begin{pmatrix} 1 & 0 \\ 0 & 1 \end{pmatrix} , \quad
        Y_2 = \begin{pmatrix} 0 & -1 \\ 1 & 0 \end{pmatrix}.
    \end{equation*}
We prove that the MLE given $Y$ is unique although the stabilizer of $Y$ is infinite.

We first show that $Y$ is polystable under the left-right action of $\SL_2 \times \SL_2$. Note that any matrix in $\SL_2$ has Frobenius norm at least $\sqrt{2}$. Indeed, if $\sigma_1$ and $\sigma_2$ are the singular values of $g$, then $\| g \|^2 = \sigma_1^2 + \sigma_2^2$, where $\sigma_1 \sigma_2 = 1$. By the arithmetic mean - geometric mean inequality, we have $\| g \|^2 \geq 2$. Therefore $Y_1$ and $Y_2$ have minimal Frobenius norm in $\SL_2$ and thus $Y$ is of minimal norm in its orbit. By Kempf-Ness, Theorem~\ref{thm:KempfNess}(d), the matrix tuple $Y$ is polystable.

The stabilizer of $Y$ consists of matrices $(g_1,g_2) \in \SL_2 \times \SL_2$ with $g_1 Y_i g_2\T = Y_i$. For $Y_1$, this gives $g_1 g_2\T = I_2$, i.e. $g_2\T = g_1^{-1}$.  Then, from $Y_2$, we obtain $g_1 Y_2 = Y_2 g_1$, and so 
    \begin{equation*}
        g_1 = \begin{pmatrix} a & b\\ -b & a \end{pmatrix} \; \text{ with } \; a^2 + b^2 = 1,
    \end{equation*}
i.e. $g_1 \in \SO_2(\RR)$ and hence $g_2 = g_1^{-\mathsf{T}} = g_1$. Thus the stabilizer of $Y$ is contained in the infinite set $\lbrace (g,g) \mid g \in \SO_2 \rbrace$. In fact, we have equality, as $\SO_2$ is commutative and $Y_1,Y_2 \in \SO_2$.

Since $Y$ is of minimal norm in its orbit, we use Theorem~\ref{thm:bigTheoremMatrixNormal} to conclude that $\lambda I_2 \otimes I_2$ is an MLE. 
Any other MLE is given by $\lambda g_1\T I_2 g_1 \otimes g_2\T I_2 g_2$ for some $(g_1,g_2)$ in the stabilizer of $Y$. Since the stabilizer is contained in $\SO_2 \times \SO_2$, the MLE is unique.

We remark that for the complex matrix normal model the MLEs involve $g^* \! g$ rather than $g\T \! g$, by Proposition~\ref{prop:MLEsViaStabilizerComplex}, hence from the complex stabilizer $\lbrace (g,g) \mid g \in \SO_2(\CC) \rbrace$ we obtain infinitely many MLEs.
\hfill\exSymbol
\end{example}

The following example shows that all stability conditions in Theorem~\ref{thm:bigTheoremMatrixNormal}(a)--(d) can occur.

\begin{example} We set $m_1=m_2=2$, and study stability under $\SL_2 \times \SL_2$ on $(\RR^{2 \times 2})^n$. We use the matrices
    \begin{equation*}
        Y_1 = \begin{pmatrix} 1 & 0 \\ 0 & 1 \end{pmatrix}, \quad
        Y_2 = \begin{pmatrix} 0 & -1 \\ 1 & 0 \end{pmatrix}, \quad
        Y_3 = \begin{pmatrix} 0 & 1 \\ 1 & 0 \end{pmatrix}, \quad
        Y_4 = \begin{pmatrix} 0 & 1 \\ 0 & 0 \end{pmatrix}.
    \end{equation*}
\begin{itemize}
    \item[(a)] The matrix $Y_4$ is unstable and the matrix tuple $(Y_4,Y_4)$ is unstable as well.
    
    \item[(b)] The orbit of the matrix tuple $(Y_1,Y_4)$ is contained in $\lbrace (g,M) \mid g \in \SL_2, \, M \neq 0 \rbrace$. In particular, $(Y_1,Y_4)$ is semistable as $\SL_2$ is closed. Moreover, for any $g \in \SL_2$ and $M \in \RR^{2 \times 2} \setminus \lbrace 0 \rbrace$ we have
        \begin{equation*}
            \| (g,M) \|^2 = \| g \|^2 + \| M \|^2 \geq 2 + \| M \|^2 > 2, 
        \end{equation*}
    where we used $\|g\|^2 \geq 2$, see Example~\ref{ex:PolystableNotStableUniqueMLE}. On the other hand, we have
        \begin{equation*}
            \left( \begin{pmatrix} \varepsilon & 0 \\ 0 & \varepsilon^{-1} \end{pmatrix},
            \begin{pmatrix} \varepsilon^{-1} & 0 \\ 0 & \varepsilon \end{pmatrix} \right) \cdot (Y_1,Y_4) 
            = \left( \begin{pmatrix} 1 & 0 \\ 0 & 1 \end{pmatrix},
            \begin{pmatrix} 0 & \varepsilon^2 \\ 0 & 0 \end{pmatrix}\right),
        \end{equation*}
    which tends to $(Y_1,0)$ as $\varepsilon \to 0$. Since $\| (Y_1,0) \|^2 = 2$ the capacity of $(Y_1,Y_4)$ is not attained by an element in the orbit of $(Y_1,Y_4)$, and $Y$ is not polystable.
    
    \item[(c)] The matrix $Y_1=I_2$ is polystable by Kempf-Ness, Theorem~\ref{thm:KempfNess}(d), as it is an $\SL_2$ matrix of minimal norm. An MLE is given by $\lambda I_2 \otimes I_2$, where $\lambda$ is the minimizer of the outer infimum in \eqref{eqn:doubleinf}. Furthermore, $Y_1$ is not stable, because its stabilizer is $\lbrace (g,g^{-\mathsf{T}}) \mid g \in  \SL_2 \rbrace$. There are infinitely many MLEs given $Y$, of the form $\lambda g\T g \otimes g^{-1} g^{-\mathsf{T}}$ for $g \in \SL_2$, see Theorem~\ref{thm:bigTheoremMatrixNormal}.
    
    \item[(d)] We show that $Y = (Y_1,Y_2,Y_3)$ is stable. First, any tuple $(M_1,M_2,M_3)$ in the orbit of $Y$ satisfies $M_1,M_2 \in \SL_2$ and $\det(M_3) = -1$. Any $2 \times 2$ matrix of determinant $\pm 1$ has Frobenius norm at least $\sqrt{2}$, by the same argument as in Example~\ref{ex:PolystableNotStableUniqueMLE}. Therefore, $Y$ is of minimal norm in its orbit, and hence polystable by Theorem~\ref{thm:KempfNess}(d). It remains to show that the stabilizer of $Y$ is finite. The discussion from Example~\ref{ex:PolystableNotStableUniqueMLE} ensures that the stabilizer of $Y$ is contained in $ \lbrace (g,g) \mid g \in \SO_2 \rbrace$. Given $g \in \SO_2$, the condition $g Y_3 g\T = Y_3$ implies $g Y_3 = Y_3 g$. This holds exactly for $g = \pm I_2$. Therefore, the stabilizer of $Y$ is the finite set $\lbrace (I_2,I_2) , (-I_2,-I_2) \rbrace$. \hfill\exSymbol
\end{itemize}
\end{example}

\subsection{Boundedness of the likelihood via semistability}

We give new conditions that guarantee the boundedness of the likelihood in a matrix normal model. To do this, we use the equivalence of the boundedness of the likelihood with the semistability of a matrix tuple under left-right action, see Theorem~\ref{thm:bigTheoremMatrixNormal}(b). 
We consider matrix tuples in $(\RR^{m_1 \times m_2})^n$ where we may assume by duality that $m_1 \geq m_2$.
The null cone of the complex left-right action of $\SL_{m_1}(\CC) \times \SL_{m_2}(\CC)$ on matrix tuples was described in~\cite[Theorem~2.1]{BurginDraisma}.
We prove the real analogue of this result and, 
with this, give a characterization of the matrix tuples with unbounded log-likelihood in Theorem~\ref{thm:nullconeLeftRight}.
This has been derived in~\cite[Theorems 3.1(i) and 3.3(i)]{DrtonKuriki} using a different method. 

The dimension of the complex null cone is given in~\cite{BurginDraisma}. 
By translating this result to the real numbers,
we derive a new upper bound on the maximum likelihood threshold $\mltb$, the minimum number of samples needed for the likelihood function to be generically bounded from above; see Corollary~\ref{cor:newMLEbound}. 
This translates in invariant theory to finding the minimum sample size $n$ such that the null cone does not fill its ambient space. 
In addition, we recover lower and upper bounds from the literature in Corollaries \ref{cor:knownMLEbound}, \ref{cor:newMLEboundWeaker} and \ref{cor:divisible}. 

\begin{theorem}
\label{thm:nullconeLeftRight}
Consider $Y \in (\RR^{m_1 \times m_2})^n$, a tuple of $n$ samples from a matrix normal model.
The log-likelihood function $\ell_Y$ is not bounded from above if and only if
there exist subspaces $V_1 \subseteq \RR^{m_1}$ and $V_2 \subseteq \RR^{m_2}$ with
$m_1 \dim V_2 > m_2 \dim V_1$
such that
$Y_i V_2 \subseteq V_1$ for all $i = 1, \ldots, n$.
\end{theorem}

\begin{proof}
The log-likelihood $\ell_Y$ is bounded from above if and only if $Y$ is not in the complex null cone under the left-right action of $\SL_{m_1}(\CC) \times \SL_{m_2}(\CC)$, by Theorem~\ref{thm:bigTheoremMatrixNormal}(b) and Proposition~\ref{prop:realVsComplexCapacity}. 
The latter is equivalent to the existence of subspaces $W_1 \subseteq \CC^{m_1}$ and $W_2 \subseteq \CC^{m_2}$ with $m_1 \dim_\CC W_2 > m_2 \dim_\CC W_1$ such that $Y_i W_2 \subseteq W_1$ for all $i=1,\ldots,n$, by \cite[Theorem 2.1]{BurginDraisma}.
This is the same condition as in the statement, except with complex subspaces. 
The real condition directly implies the complex one.
We now show the reverse implication, following an argument thanks to Jan Draisma.

Given complex subspaces $W_1 \subseteq \CC^{m_1}$ and $W_2 \subseteq \CC^{m_2}$ as above, let $V_j$ be the intersection of $W_j$ with $\RR^{m_j}$, and let $V_j'$ be the image of $W_j$ under the map that sends a complex vector to its real part.
Since $\mathsf{i}  V_j$ is the kernel of that map, where $\mathsf{i}^2 = -1$, we have $2 \dim_\CC W_j = \dim_\RR V_j + \dim_\RR V'_j$.
In particular, we either have $m_1 \dim V_2 > m_2 \dim V_1$ or $m_1 \dim V'_2 > m_2 \dim V'_1$.
Since both inclusions $Y_i V_2 \subseteq V_1$ and $Y_i V'_2 \subseteq V'_1$ hold for all $i=1, \ldots, n$,
either $(V_1, V_2)$ or $(V'_1, V'_2)$ are real subspaces  as in the statement.
\end{proof}

We now come to statistical implications of Theorem~\ref{thm:nullconeLeftRight}.

\begin{corollary}
\label{cor:knownMLEbound}
If $n < \frac{m_1}{m_2}$, then the log-likelihood function $\ell_Y$ is unbounded from above for every tuple of samples $Y \in (\RR^{m_1 \times m_2})^n$.
In particular, $\mltb(m_1,m_2) \geq \lceil \frac{m_1}{m_2} \rceil$.
\end{corollary}

\begin{proof}
For any one-dimensional subspace $V_2 \subseteq \RR^{m_2}$,
the dimension of $V_1 := \sum_{i=1}^n Y_i V_2$ is at most $n$.
If $n < \frac{m_1}{m_2}$, Theorem~\ref{thm:nullconeLeftRight} implies that the log-likelihood $\ell_Y$ is unbounded.
\end{proof}

The result in this corollary also follows from \cite[Lemma 1.2]{DrtonKuriki}. 
We now characterize when the null cone fills the space of matrix tuples, which extends~\cite[Proposition 2.4]{BurginDraisma} from the space of complex matrix tuples to real matrix tuples.
For this, we begin by defining the cut-and-paste rank from~\cite[Definition 2.2]{BurginDraisma} over the real numbers.

\begin{definition}
\label{def:cprank}
The \emph{cut-and-paste rank} $\cp^{(n)}(a,b,c,d)$ of a tuple of positive integers $a$, $b$, $c$, $d$ and $n$ is the maximum rank of the $ab \times cd$ matrix $\sum_{i = 1}^n X_i \otimes Y_i$, as $X_i$ and $Y_i$ range over real matrices of sizes $c \times a$ and $d \times b$ respectively.
\end{definition}

\begin{remark}\label{rem:RealVsComplexCPrank}
Analogously to Definition~\ref{def:cprank} one can define $\cp^{(n)}_{\CC}(a,b,c,d)$ by letting the $X_i$ and $Y_i$ range over complex matrices, see \cite[Definition~2.2]{BurginDraisma}. The real and complex ranks agree, as follows. The condition for the rank of the matrix $\sum_{i = 1}^n X_i \otimes Y_i$ to drop is given by minors. Thus, $\cp^{(n)}_{\CC}(a,b,c,d)$ is witnessed on a Zariski-open subset of $W := (\CC^{c\times a})^n \times (\CC^{d\times b})^n$ and hence witnessed by some element in $(\RR^{c\times a})^n \times (\RR^{d\times b})^n$, as the latter is Zariski-dense in $W$.
\end{remark}
 
We use the cut-and-paste rank to give a necessary and sufficient  condition for the null cone under left-right action to fill the space of matrix tuples $(\RR^{m_1 \times m_2})^n$, i.e. for the log-likelihood to be always unbounded from above. As above, we take $m_1 \geq m_2$. Moreover, since we saw in Corollary~\ref{cor:knownMLEbound} that the likelihood is unbounded for $m_2 n < m_1$, it suffices to restrict to the range $m_2 \leq m_1 \leq n m_2$.

\begin{theorem}
\label{thm:nullconeFills}
Let $0 < m_2 \leq m_1 \leq n m_2$.
The log-likelihood $\ell_Y$ is unbounded from above for every tuple of samples $Y \in (\RR^{m_1 \times m_2})^n$
if and only if there exists $k \in \{ 1, \ldots, m_2\}$ such that $l = \lceil \frac{m_1}{m_2} k \rceil - 1$ satisfies both
\begin{align*}
m_1-l \leq n(m_2-k) \quad &\text{ and}   \\
\cp^{(n)}(a,b,c,d) = cd, \quad  &\text{ where} \quad
 (a,b,c,d) = (m_2-k,k,m_1-l,nk-l).
\end{align*}
\end{theorem}

\begin{proof}
Let $\mathcal{N}_{\KK}$ be the null cone under the left-right action of $\SL_{m_1}(\KK) \times \SL_{m_2}(\KK)$ on $(\KK^{m_1 \times m_2})^n$, where $\KK \in \{ \RR, \CC\}$. We note that $\mathcal{N}_{\CC}$ is Zariski closed and that $(\RR^{m_1 \times m_2})^n$ is Zariski-dense in $(\CC^{m_1 \times m_2})^n$. Thus, $\mathcal{N}_{\RR}$ fills the space $(\RR^{m_1 \times m_2})^n$ if and only if $\mathcal{N}_{\CC}$ fills the space $(\CC^{m_1 \times m_2})^n$, by Proposition~\ref{prop:realVsComplexCapacity}. It therefore suffices to characterize when $\mathcal{N}_{\CC} = (\CC^{m_1 \times m_2})^n$. For this, define for natural numbers $k$ and $l$
\begin{equation*}
    Q_{k,l} := \left\lbrace (Y_1, \ldots, Y_n) \in (\CC^{m_1 \times m_2})^n \mid 
    \exists V \subseteq \CC^{m_2}: 
    \dim_{\CC} V = k,
    \dim_{\CC}(\sum_{i=1}^n Y_i V ) \leq l
    \right\rbrace.
\end{equation*}
The null cone $\mathcal{N}_{\CC}$ is the union of the $Q_{k,l}$ over $1 \leq k \leq m_2$ and $0 \leq l < \frac{m_1}{m_2} k$, 
by \cite[Theorem~2.1]{BurginDraisma}, which is the complex analogue of Theorem~\ref{thm:nullconeLeftRight}. We observe that the algebraic sets $Q_{k,l}$ get larger as $l$ increases. Hence, it suffices to consider if any of the $Q_{k,l}$ fills $(\CC^{m_1 \times m_2})^n$ as $k$ ranges over $1 \leq k \leq m_2$, where the corresponding $l$ is the largest integer strictly smaller than $\frac{m_1}{m_2}k$, i.e. $l = \lceil \frac{m_1}{m_2} k \rceil - 1$.

The assumption $m_1 \leq nm_2$ yields $l < nk$. Therefore, \cite[Proposition 2.4]{BurginDraisma} shows that
\begin{align*}
    \dim_{\CC} Q_{k,l} = n m_1 m_2 - \left( (m_1-l)(kn-l)- \cp^{(n)}_{\CC}(a,b,\tilde{c},d)  \right),
\end{align*}    
where $a = m_2 - k$, $b=k$, $\tilde{c} = \min \{ m_1 - l, n (m_2-k) \}$ and $d = kn-l$. By Remark~\ref{rem:RealVsComplexCPrank}, $\cp^{(n)}_{\CC}(a,b,\tilde{c},d) = \cp^{(n)}(a,b,\tilde{c},d)$. Thus, $Q_{k,l}$ equals $(\CC^{m_1 \times m_2})^n$ if and only if    
\begin{equation*}
    \cp^{(n)}(a,b,\tilde{c},d) = (m_1-l)(kn-l).
\end{equation*}
Finally, the latter equation is equivalent to
\begin{equation*}
    m_1 - l \leq n (m_2-k) \quad \text{ and } \quad \cp^{(n)}(a,b,\tilde{c},d) = \tilde{c}d,
\end{equation*}
since $\tilde{c} = \min \{ m_1 - l, n (m_2-k) \}$, $d = kn-l \geq 1$ and $\cp^{(n)}(a,b,\tilde{c},d) \leq \tilde{c}d$.
\end{proof}

In principle, Theorem~\ref{thm:nullconeFills} solves the problem of determining the maximum likelihood threshold $\mltb$, although in terms of the cut-and-paste rank. Hence, this gives statistical motivation for better understanding the cut-and-paste rank, e.g. by obtaining a general closed formula.

We use the above theorem to give a new upper bound for $\mltb$.

\begin{corollary}
\label{cor:newMLEbound}
Let $0 < m_2 \leq m_1$.  If 
\begin{equation}
\label{eq:newBound}
    n > \max_{1\leq k \leq m_2} \left(\frac{l}{k} + \frac{m_2-k}{m_1 - l}\right), \quad \text{ where } l = \left\lceil \frac{m_1}{m_2} k \right\rceil - 1,
\end{equation}
the log-likelihood $\ell_Y$
for a generic matrix tuple $Y \in (\RR^{m_1 \times m_2})^n$ is bounded from above. In other words, $\mltb \leq \left\lfloor \max \limits_{1\leq k \leq m_2} \left(\frac{l}{k} + \frac{m_2-k}{m_1 - l}\right) \right\rfloor +1$.
\end{corollary} 

\begin{proof}
First, we observe that \eqref{eq:newBound} with $k=m_2$ yields
$n > \frac{m_1-1}{m_2}$. 
The latter is equivalent to $nm_2 \geq m_1$, so we are in the setting of Theorem~\ref{thm:nullconeFills}.
Using the notation in that theorem, we see that~\eqref{eq:newBound} is equivalent to every $k \in \{1, \ldots, m_2 \}$ satisfying
$cd > ab$.
In particular, for every such $k$ we have
$\cp^{(n)}(a,b,c,d) \leq ab < cd$, 
so by Theorem~\ref{thm:nullconeFills} the log-likelihood $\ell_Y$ cannot be unbounded from above for every tuple $Y$.
\end{proof}

Two simpler upper bounds, which are known in the statistics literature~\cite[Proposition~1.3, Theorem~1.4]{DrtonKuriki}, are obtained as follows.

\begin{corollary}
\label{cor:newMLEboundWeaker}
If  $n \geq \frac{m_1}{m_2} + \frac{m_2}{m_1}$, then the log-likelihood $\ell_Y$
for a generic matrix tuple $Y \in (\RR^{m_1 \times m_2})^n$ is bounded from above. In other words, $\mltb \leq \lceil \frac{m_1}{m_2} + \frac{m_2}{m_1} \rceil$.
\end{corollary}

\begin{proof}
For every $k \in \{1, \ldots, m_2 \}$ we have $l < \frac{m_1k}{m_2}$, which implies that
\begin{equation*}
    \frac{m_1}{m_2} + \frac{m_2}{m_1} > 
    \frac{l}{k} + \frac{m_2-k}{m_1-l}.
\end{equation*}
Thus, the assertion follows from Corollary~\ref{cor:newMLEbound}.
 \end{proof}
 
\begin{corollary}
\label{cor:divisible}
Let $m_2$ divide $m_1$. 
The log-likelihood $\ell_Y$
for a generic matrix tuple $Y \in (\RR^{m_1 \times m_2})^n$ is bounded from above
if and only if $n \geq \frac{m_1}{m_2}$.
In other words, $\mltb = \frac{m_1}{m_2}$.
\end{corollary}

\begin{proof}
If $n < \frac{m_1}{m_2}$, the log-likelihood is always unbounded from above by Corollary~\ref{cor:knownMLEbound}.
So we write $m_1 = \gamma m_2$ and assume $n \geq \gamma$.
For every $k \in \{1, \ldots, m_2 \}$, 
using the notation from Theorem~\ref{thm:nullconeFills}, we see that $l = \gamma k -1$ and $a < c$.
If $n > \gamma$, we also have that $b < d$, so
$\cp^{(n)}(a,b,c,d) \leq ab < cd$.
If $n=\gamma$, then $m_1 - l > n(m_2-k)$.
In either case, one of the two conditions in Theorem~\ref{thm:nullconeFills} is not satisfied, so $\ell_Y$ is generically bounded from above.
\end{proof}
 
 \begin{table}[h!tb]
 \caption{\label{tab:mlt} Bounds for the maximum likelihood threshold $\mltb$.
    $L = \lceil \frac{m_1}{m_2} \rceil$ is the lower-bound from Corollary~\ref{cor:knownMLEbound}, $U=\lceil \frac{m_1}{m_2} + \frac{m_2}{m_1} \rceil$ is the upper bound from Corollary~\ref{cor:newMLEboundWeaker}, and $\alpha$ is our new upper bound from Corollary~\ref{cor:newMLEbound}.}
 \begin{small}
 \begin{minipage}[t]{.33\linewidth}
      \centering
        \begin{tabular}[t]{ !{\vrule width 1pt} c  c | c  c  c  c !{\vrule width 1pt}} \Xhline{1pt}
        $m
        _1$ & $m_2$ & $L$ & $\mltb$  & $\alpha$  & $U$\\ 
				\hline 2 & 2 & 1 & 1 & 1 & 2 \\
				\hdashline
				 3 & 2 & 2 & 2 & 2 & 3 \\
				 3 & 3 & 1 & 1 & 2 & 2  \\
				 \hdashline
				 4 & 2 & 2 & 2 & 2 & 3 \\
				 4 & 3 & 2 & 2 & 2 & 3 \\
				 4 & 4 & 1 & 1 & 2 & 2 \\
				 \hdashline
				 5 & 2 & 3 & 3 & 3 & 3 \\
				 5 & 3 & 2&  3&  3&  3\\
				 5 & 4 &  2& 2 & 2 & 3 \\
				 5 & 5 & 1 & 1 & 2 & 2 \\
				 \hdashline
				 6 & 2 & 3 &3  & 3 & 4 \\
				 6 & 3 & 2 & 2 & 2 & 3 \\
				 6 & 4 & 2 & 2 & 2 & 3 \\
				 6 & 5 & 2 & 2 & 2 & 3 \\
				 6 & 6 & 1 & 1 & 2 & 2 \\
			\Xhline{1pt}
				 \end{tabular}
				    \end{minipage} \begin{minipage}[t]{.33\linewidth}
      \centering
        \begin{tabular}[t]{ !{\vrule width 1pt} c  c | c  c  c  c !{\vrule width 1pt}} \Xhline{1pt}
        $m
        _1$ & $m_2$ & $L$ & $\mltb$  & $\alpha$  & $U$\\ 
				\hline 
				 7 & 2 & 4 & 4 & 4  & 4 \\
			     7 & 3 & 3 & 3 & 3 & 3 \\
			     7 & 4& 2 &  3&  3&  3\\
			     7 & 5 & 2 & 3 & 3 & 3 \\
			     7 & 6 & 2 & 2 & 2 & 3 \\
			     7 & 7 & 1 & 1 & 2 & 2 \\
			     \hdashline
			      8 & 2 & 4 & 4 & 4 & 5 \\
			     8 & 3 & 3 & 3 & 3 & 4 \\
			     8 & 4 & 2 & 2 & 3 & 3 \\
			     8 & 5 & 2 & 3 & 3 & 3 \\
			     8 & 6 & 2 & 2 & 2 & 3 \\
			     8 & 7 & 2 & 2 & 2 & 3 \\
			     8 & 8 & 1 & 1 & 2 & 2 \\
			    \hdashline 9 & 2 & 5 & 5 & 5 & 5  \\
			    9 & 3 & 3 & 3 & 3 & 4 \\
			\Xhline{1pt}
				 \end{tabular}
				    \end{minipage}%
    \begin{minipage}[t]{.33\linewidth}
      \centering
       \begin{tabular}[t]{  !{\vrule width 1pt} c  c | c c  c  c  !{\vrule width 1pt}} \Xhline{1pt}  $m
        _1$ & $m_2$ & $L$ & $\mltb$  & $\alpha$  & $U$\\ 	
                \hline 
			     
			     9 & 4 & 3 & 3 & 3 & 3 \\
			     9 & 5 & 2 & 3 & 3 & 3 \\
			     9 & 6 & 2 & 2 & 2 & 3 \\
			     9 & 7 & 2 & 3 & 3 & 3 \\
			     9 & 8 & 2 & 2 & 2 & 3 \\
			     9 & 9 & 1 & 1 & 2 & 2 \\
			    \hdashline 10 & 2 & 5 & 5 & 5 & 6  \\
			     10 & 3 & 4 & 4 & 4 & 4 \\
			     10 & 4 & 3 & 3 & 3 & 3 \\
			     10 & 5 & 2 & 2 & 3 & 3 \\
			     10 & 6 & 2 & 3 & 3 & 3 \\
			     10 & 7 & 2 & 3 & 3 & 3 \\
			     10 & 8 & 2 & 2 & 2 & 3 \\
			     10 & 9 & 2 & 2 & 2 & 3 \\
			     10 & 10 & 1 & 1 &2 & 2 \\
				\Xhline{1pt} \end{tabular}
				    \end{minipage}
				    \end{small}
\end{table}

In Table \ref{tab:mlt} we list the maximum likelihood threshold $\mltb$ for boundedness  of the log-likelihood for small values of $m_1, m_2$, and compare with the bounds discussed above. 
We observe that there are cases where our upper bound 
\[\alpha = \left\lfloor \max \limits_{1\leq k \leq m_2} \left(\frac{l}{k} + \frac{m_2-k}{m_1 - l}\right) \right\rfloor +1, \quad \text{ where } l = \left\lceil \frac{m_1}{m_2} k \right\rceil - 1,
\]
is strictly better than the simple upper bound $U =\lceil \frac{m_1}{m_2} + \frac{m_2}{m_1} \rceil$, e.g.  when $(m_1,m_2)=(3,2)$. 
In most cases our bound $\alpha$ matches the lower bound $L=\lceil \frac{m_1}{m_2} \rceil$, so that we can determine $\mlt_b$.
In addition, when $m_2 | m_1$, one can use Corollary \ref{cor:divisible} to determine $\mltb$ even if the bounds $L$ and $\alpha$ do not coincide, such as in $(m_1,m_2)=(8,4)$ or in the square cases $m_1=m_2$. The rest of the values of $\mltb$ can be filled from \cite[Table 1]{DrtonKuriki}. 
We highlight the case $(m_1, m_2) = (8,3)$:
the maximum likelihood threshold $\mltb = 3$ was computed in~\cite{DrtonKuriki} via Gr\"obner bases, 
but it is not covered by the general bounds in~\cite{DrtonKuriki}.
Nevertheless, our bound $\alpha$ determines this case.

\subsection{Uniqueness of the MLE via stability}

We compare conditions for stability with conditions for the uniqueness of the MLE. We saw in Example~\ref{ex:PolystableNotStableUniqueMLE} that stability of a real matrix tuple $Y$ under left-right action of $\SL_{m_1}(\RR) \times \SL_{m_2}(\RR)$ is not equivalent to uniqueness of the MLE given $Y$. 
However, such an equivalence holds for complex Gaussian models, by  Theorem~\ref{thm:MLEversusStabilityComplex}.
Matrix normal models over the complex numbers are induced by the left-right action of $\SL_{m_1}(\CC) \times \SL_{m_2}(\CC)$ on $(\CC^{m_1 \times m_2})^n$.
Hence we obtain conditions for the uniqueness of the MLE given $Y \in (\CC^{m_1 \times m_2})^n$ from characterizing the stability of $Y$ under the left-right action by $\SL_{m_1}(\CC) \times \SL_{m_2}(\CC)$.
Characterizing this stability is a special case of the setting studied in~\cite{King}.
From this, we obtain the following theorem, which we prove in Appendix~\ref{sec:appendix}. 

\begin{theorem}
\label{thm:king}
Consider the left-right action of $\SL_{m_1}(\CC) \times \SL_{m_2}(\CC)$ on $(\CC^{m_1 \times m_2})^n$,
and a tuple $Y \in (\CC^{m_1 \times m_2})^n$ of $n$ samples from a complex matrix normal model. 
The following are equivalent:
\begin{itemize}
    \item[(a)] the complex MLE given $Y$ exists uniquely;
    \item[(b)] the matrix tuple $Y$ is stable; 
    \item[(c)] the matrix $(Y_1 | \ldots | Y_n) \in \CC^{m_1 \times n m_2}$ has rank $m_1$, and
    $
        m_2 \dim V_1 > m_1 \dim V_2
    $
holds for all subspaces  $V_1 \subseteq \CC^{m_1}$, $\lbrace 0 \rbrace \subsetneq V_2 \subsetneq \CC^{m_2}$ that satisfy $Y_i V_2 \subseteq V_1$ for all $i=1,\ldots,n$.
\end{itemize}
\end{theorem}

We note the similarity with the conditions  that characterize semistability in Theorem~\ref{thm:nullconeLeftRight}. 
However, while Theorem~\ref{thm:nullconeLeftRight} 
holds both over $\RR$ and $\CC$, the same cannot be true for Theorem~\ref{thm:king}.
In fact, the real analog of Theorem~\ref{thm:king}(c) is shown to characterize uniqueness of the MLE in~\cite[Theorems~3.1(ii) and 3.3(ii)]{DrtonKuriki}, 
which is not equivalent to stability by Example~\ref{ex:PolystableNotStableUniqueMLE}.

\subsection{The moment map}
In this section we recall the condition for the moment map for the action of $\SL_{m_1} \times \SL_{m_2}$ to vanish at a matrix tuple. By Kempf Ness, Theorem~\ref{thm:KempfNess}(a), this gives the condition to be at a point of minimal norm in the orbit.

The tangent space of $\SL_{m_i}$ at the identity matrix consists of all matrices with trace zero.
The moment map at $Y \in (\RR^{m_1 \times m_2})^n$
is the differential of 
$(g_1, g_2) \mapsto \Vert (g_1, g_2) \cdot Y \Vert^2$
at the pair $(I_{m_1}, I_{m_2})$ of identity matrices, i.e.
\begin{align}
\label{eq:momentMapMatrixNormal}    \begin{split}
    \mu(Y): T_{I_{m_1}} \SL_{m_1} \times T_{I_{m_2}} \SL_{m_2} &\longrightarrow  \RR \\
    (\dot g_1, \dot g_2) &\longmapsto  
    2 \sum_{i=1}^n \tr \left( (\dot g_1 Y_i + Y_i \dot g_2\T) Y_i\T \right).
\end{split}
\end{align}

\begin{theorem}[Kempf-Ness theorem for $\SL \times \SL$ action]
\label{thm:KNmatrixNormal}
Consider the left-right action of $\SL_{m_1}\times\SL_{m_2}$ on the space of matrix tuples $(\RR^{m_1 \times m_2})^n$.
A matrix tuple is semistable (resp. polystable) if and only if there is a non-zero matrix tuple $Y$ in its orbit closure (resp. orbit) where the moment map $\mu$ vanishes, i.e.
\begin{equation*}
    \exists \, c_1, c_2 > 0 : 
    \sum_{i=1}^n Y_i Y_i\T = c_1 I_{m_1} \text{ and }
    \sum_{i=1}^n Y_i\T Y_i = c_2 I_{m_2}.
\end{equation*}
\end{theorem}

\begin{proof}
This follows from rewriting~\eqref{eq:momentMapMatrixNormal} as
\begin{align*}
\hspace*{30mm}
    (\dot{g}_1,\dot{g}_2) \longmapsto
    2\, \tr \left( \dot{g}_1 \sum_{i=1}^n Y_i Y_i\T \right) + 2 \,\tr \left( \dot{g}_2\T \sum_{i=1}^n Y_i\T Y_i \right).
    \hspace*{30mm}
    \qedhere
\end{align*}
\end{proof}

\subsection{Scaling algorithms for the MLE}
\label{sec:matrixnormalalg}
In this section, we describe algorithmic consequences of the connection between invariant theory and maximum likelihood estimation. 
We present an algorithm for ML estimation that is well-known in statistics, and 
connect it to an algorithm in invariant theory; see the left hand side of Figure~\ref{fig:algorithms}. The connection allows us to give a complexity analysis of the statistics algorithm.
The algorithm in statistics is the flip-flop algorithm, which involves the group $\GL_{m_1} \times \GL_{m_2}$, while the invariant theory algorithm is operator scaling for the left-right action of $\SL_{m_1} \times \SL_{m_2}$.
We begin by recalling these algorithms.

\subsubsection{Operator scaling and the flip-flop algorithm} Operator scaling, see the top left in Figure~\ref{fig:algorithms}, solves the norm minimization problem for the left-right action of $\SL_{m_1}(\CC) \times \SL_{m_2}(\CC)$ on the space of matrix tuples $(\CC^{m_1 \times m_2})^n$. 
From an invariant theory perspective, operator scaling was first studied in \cite{GurvitsOperatorScaling}, and \cite{GGOWoperatorScaling} showed that it yields a polynomial time algorithm for null cone membership. The method was generalized to tuples of tensors in \cite[Algorithm~1]{PeterAvi1}. 

The \emph{flip-flop algorithm}~\cite{dutilleul1999mle,lu2005likelihood}, see the bottom left of Figure~\ref{fig:algorithms}, is an alternating maximization procedure
to find an MLE in a matrix normal model.
It can be thought of as a Gaussian version of IPS for matrix normal models, since one alternatingly updates the estimates in each marginal.
If we consider $\Psi_2$ to be fixed, the log-likelihood in~\eqref{eqn:matrixnormallikelihood} becomes, up to constants,  
\begin{equation*}
    \ell_Y(\Psi_1) = m_2 \left[ 
    \log \det(\Psi_1) - \tr \left( 
    \Psi_1 \cdot \frac{1}{nm_2} \sum_{i=1}^n Y_i \Psi_2 Y_i\T
    \right)
    \right].
\end{equation*}
Maximizing the log-likelihood with respect to $\Psi_1$ reduces to the case of a standard multivariate Gaussian model as in~\eqref{eqn:gaussianlikelihood}.
The unique maximizer over the positive definite cone is the inverse, if it exists, of the matrix $\frac{1}{nm_2} \sum_{i=1}^n Y_i \Psi_2 Y_i\T$.
In the same way, we can fix $\Psi_1$ and maximize the log-likelihood with respect to $\Psi_2$. Iterating these two steps gives the algorithm.

\begin{algorithm}
\caption{Flip-flop}
\label{alg:flipflop}
\begin{algorithmic}[1]
\REQUIRE{$Y_1, \ldots, Y_n \in \RR^{m_1 \times m_2}$, $N \in \ZZ_{>0}$.}
\ENSURE{an approximation of an MLE, if it exists.}
\STATE{Initialize $\Psi_2 := I_{m_2}$.}
\FOR{$k=1$ \TO $N$}
\STATE the following pair of updates
\begin{align}
     \label{eq:updateMatrixNormal}
        \begin{split}
                     \Psi_1 &:= \left( \frac{1}{nm_2} \sum_{i=1}^n Y_i \Psi_2 Y_i\T \right)^{-1} \\
         \Psi_2 &:= \left( \frac{1}{nm_1} \sum_{i=1}^n Y_i\T \Psi_1 Y_i \right)^{-1}.
        \end{split}
     \end{align}
\ENDFOR
\RETURN $\Psi_1 \otimes \Psi_2$.
\end{algorithmic}
\end{algorithm}

We now compare operator scaling with the flip-flop algorithm. The scaling algorithm in~\cite[Algorithm 1]{PeterAvi1} gives, when specializing from tensors to matrices, the same procedure as Algorithm~\ref{alg:flipflop}, up to scaling with different constants in the update steps~\eqref{eq:updateMatrixNormal}.
In~\cite[Algorithm~1]{PeterAvi1}, the matrices $\Psi_1$ and $\Psi_2$ in~\eqref{eq:updateMatrixNormal} are restricted to have determinant one, in order to stay in the $\SL_{m_1} \times \SL_{m_2}$ orbit of $Y$. In comparison, Algorithm~\ref{alg:flipflop} has constants chosen to minimize the outer infimum in~\eqref{eqn:doubleinf}.

Although the algorithm in~\cite{PeterAvi1} is defined over the complex numbers,
when restricting to real inputs operator scaling only involves computations over the reals. This allows the computation of MLEs (if they exist) in the real matrix normal model via \eqref{eqn:doubleinf}, since the capacity of a real matrix tuple is the same under the action of $\SL_{m_1}(\RR) \times \SL_{m_2}(\RR)$ as under the action of $\SL_{m_1}(\CC) \times \SL_{m_2}(\CC)$, see Proposition~\ref{prop:realVsComplexCapacity}.

\subsubsection{Convergence}
In~\cite{PeterAvi1}, the authors give conditions for being in the null cone, based on the convergence of their Algorithm 1.
Specializing to a matrix tuple, to connect to the flip-flop algorithm, their results combine with ours to show the following. If an update step cannot be computed because one of the matrices in~\eqref{eq:updateMatrixNormal} cannot be inverted, then the matrix tuple $Y$ is unstable under the  action of $\SL_{m_1}(\CC) \times \SL_{m_2}(\CC)$, and therefore also under the real action of $\SL_{m_1}(\RR) \times \SL_{m_2}(\RR)$, by Proposition~\ref{prop:realVsComplexCapacity}. This implies that the log-likelihood $\ell_Y$ is unbounded, by Theorem~\ref{thm:bigTheoremMatrixNormal}(a).
Otherwise, the sequence of terms
$(\Psi_1^{\nicefrac{1}{2}}, \Psi_2^{\nicefrac{1}{2}}) \cdot Y$ converges, possibly to infinity.
We now consider the possible cases that can arise in this limit, by comparing to operator scaling, using the fact that the constants in the flip-flop algorithm minimize the outer infimum in~\eqref{eqn:doubleinf}. 
 
If the sequence  $(\Psi_1^{\nicefrac{1}{2}}, \Psi_2^{\nicefrac{1}{2}}) \cdot Y$ converges to zero or infinity, then the log-likelihood $\ell_Y$ is unbounded. Otherwise, the sequence  converges to a matrix tuple of  positive norm in the orbit closure, where the moment map~\eqref{eq:momentMapMatrixNormal} vanishes, and $Y$ is semistable. Here, two further possibilities can arise. The first possibility occurs when the matrix tuple $Y$ is polystable. Then the minimal norm is attained at an element of the group $\SL_{m_1} \times \SL_{m_2}$, and the flip-flop algorithm converges to an MLE;  see~\eqref{eqn:doubleinf}.
The second possibility occurs when $Y$ is semistable but not polystable. Then, the flip-flop algorithm diverges by the following remark.

\begin{remark}\label{rem:noExtendedMLE}
If the matrix tuple 
$Y$ is semistable but not polystable under the left-right action of $\SL_{m_1} \times \SL_{m_2}$, then the likelihood $L_Y$ (equivalently the log-likelihood $\ell_Y$) is bounded from above, but does not attain its supremum. In this case, any sequence $\Psi_N := (\Psi_{1,N} \otimes \Psi_{2,N})$ of concentration
matrices with
    \begin{align*}
        \lim_{N \to \infty} L_Y(\Psi_{1,N} \otimes \Psi_{2,N}) = \sup L_Y > 0
    \end{align*}
diverges. Indeed, otherwise the limit $\Psi_{\infty}$ would be rank-deficient, as the matrix normal model is closed in $\PD_{m_1 m_2}$. Then $\det(\Psi_\infty)=0$ yields the contradiction $\sup L_Y = L_Y(\Psi_{\infty}) = 0$.
\end{remark}

\subsubsection{Complexity} We use known results to derive a complexity analysis for the flip-flop algorithm.
In~\cite{PeterAvi1}, the authors prove convergence of their Algorithm~1, which solves the null cone membership problem up to an approximation parameter $\varepsilon > 0$. For tuples of tensors, choosing $\varepsilon$ exponentially small in the dimension of the tensor space yields a deterministic test for null cone membership with exponential running time, see \cite[Theorem~3.8]{PeterAvi1}.
When specializing to tuples of matrices, i.e. to operator scaling, it suffices to choose $\varepsilon$ polynomially small. Thus for operator scaling, \cite[Algorithm~1]{PeterAvi1} recovers the polynomial time algorithm for the null cone membership problem from \cite{GGOWoperatorScaling}. We adapt~\cite[Theorem~1.1]{PeterAvi1}
to our notation 
to derive the following.

\begin{theorem}\label{thm:flipflopcomplexity}
Given $\varepsilon > 0$ and
a matrix tuple $Y \in (\ZZ^{m_1 \times m_2})^n$
with matrix entries of bit size bounded by $b$, 
after a number of steps that is polynomial in $(nm_1m_2, b, \nicefrac{1}{\varepsilon})$,
the flip-flop algorithm either 
identifies that the log-likelihood $\ell_Y$ is unbounded
or
finds $(\Psi_1, \Psi_2) \in \PD_{m_1} \times \PD_{m_2}$
such that the matrix tuple $(\Psi_1^{\nicefrac{1}{2}}, \Psi_2^{\nicefrac{1}{2}}) \cdot Y$
is $\varepsilon$-close to a matrix tuple where the moment map~\eqref{eq:momentMapMatrixNormal} vanishes.
\end{theorem}

In the case where the log-likelihood $\ell_Y$ is bounded, taking the limit $\varepsilon \to 0$ in Theorem~\ref{thm:flipflopcomplexity} gives rise to two possibilities. Either the MLE exists and is the limit of the $\Psi_1 \otimes \Psi_2$ as $\varepsilon \to 0$, or the sequence $\Psi_1 \otimes \Psi_2$ diverges as $\varepsilon \to 0$, by Remark~\ref{rem:noExtendedMLE}.
Because of this divergence, there is no meaningful notion of approximate MLE in the latter scenario.

\subsubsection{Outlook} We briefly comment on extensions of the above to general groups, see the right hand side of Figure~\ref{fig:algorithms}.
In its full generality, the algorithm in~\cite{PeterAvi1} is an alternating minimization procedure to find the capacity of a tuple of $d$-dimensional tensors of format $m_1 \times \ldots \times m_d$ under the  action of 
$\SL_{m_1} \times \ldots \times \SL_{m_d}$.
It can therefore be used for ML estimation in (real and complex) tensor normal models.
More generally, the algorithms in~\cite{PeterAvi3} can be used for geodesically convex algorithms for maximum likelihood estimation in complex Gaussian group models as in Theorem~\ref{thm:MLEversusStabilityComplex}.
Many scaling algorithms are designed to optimize over the complex orbit, but often each update is defined over $\RR$ if the input is real, and hence they can also be used for real Gaussian group models.

\section{Transitive DAGs}  \label{sec:tdag}
In this section we study graphical models that fit into the Gaussian group model framework. We study MLE existence via a corresponding null cone problem.
We focus on directed graphs, although
our results also cover undirected graphical models, as explained in Remark~\ref{rem:undirectedGraphs}.

Let $\mathcal{G}$ be a directed acyclic graph (DAG) with $m$ nodes.
We denote an edge from $j$ to $i$ by $j \to i$; otherwise, if there is no such edge, we write $j \not\to i$.
We note that edges $i \to i$ do not appear in a DAG, because they give cycles of length one.
Consider the statistical model represented by the linear structural equation
\[ Y = \Lambda Y + \varepsilon ,\]
where $Y \in \RR^m$, the matrix $\Lambda \in \RR^{m \times m}$ satisfies $\Lambda_{ij}=0$ for $j \not\to i$ in $\mathcal{G}$, and $\varepsilon \sim N(0,\Omega)$ with $\Omega \in \RR^{m \times m}$ diagonal and positive definite. 
The model expresses each coordinate $Y_i$ as a linear combination of all $Y_j$ such that $j \to i$, up to Gaussian error.
Solving for $Y$, we have
\[ Y = (I - \Lambda)^{-1} \varepsilon, \]
where the acyclicity of $\mathcal{G}$ implies that $(I - \Lambda)$ is invertible.
We see that $Y$ is Gaussian with covariance matrix and concentration matrix
\begin{equation}
    \label{eq:graphmodel}
 \Sigma = (I - \Lambda)^{-1} \Omega (I-\Lambda)^{-\mathsf{T}} , \qquad
\Psi = (I - \Lambda)\T \Omega^{-1} (I-\Lambda).
\end{equation}
The Gaussian graphical model $\mathcal{M}_\mathcal{G}^\to$ consists of the set of concentration matrices $\Psi$ of the form in~\eqref{eq:graphmodel}, for $\Lambda$ and $\Omega$ defined in terms of $\mathcal{G}$ as above.

We now put these models in the context of Gaussian group models. Given a DAG $\mathcal{G}$, we define the set of matrices
\begin{equation}
\label{eq:GG}
 G(\mathcal{G}) = \{ g \in \GL_m \mid g_{ij}=0 \text{ for } i \neq j \text{ with } j \not \to i \text{ in }  \mathcal{G} \} .
 \end{equation}
We have a \emph{transitive} DAG (TDAG) $\mathcal{G}$ if  $k \to j$ and $j \to i$ in $\mathcal{G}$ imply $k \to i$ in $\mathcal{G}$. 

\begin{proposition}\label{prop:tdagroup}
The set of matrices $G(\mathcal{G})$ is a group if and only if $\mathcal{G}$ is a TDAG. In this case, the Gaussian graphical model given by $\mathcal{G}$ is the Gaussian group model given by $G(\mathcal{G})$:
\[ \mathcal{M}_\mathcal{G}^\to = \mathcal{M}_{G(\mathcal{G})}.\]
\end{proposition}
\begin{proof}
If $\mathcal{G}$ is not a TDAG, then there exist pairwise distinct indices $i,j,k$ such that $j\to i$ and $k\to j$ but $k \not \to i$. Take the elementary matrices $g = E_{ij}$ (with ones on the diagonal and at the $(i,j)$ entry, and zero elsewhere) and $h = E_{jk}$.
We see that $g, h \in G(\mathcal{G})$, but $gh \notin G(\mathcal{G})$ since $(gh)_{ik}=1$, hence $G(\mathcal{G})$ is not a group. 

Conversely, we assume that $\mathcal{G}$ is a TDAG.
Any invertible diagonal matrix, in particular the identity $I$, is in $ G(\mathcal{G})$. 
Suppose $g,h \in G(\mathcal{G})$ and that $(gh)_{ik} \neq 0$ for $i \neq k$. This means that there must exist some index $j$ such that $g_{ij} h_{jk} \neq 0$. 
In particular, $g_{ij} \neq 0$ and $h_{jk} \neq 0$, so that we have either $j \to i$ or $j = i$, and either $k \to j$ or $k = j$. 
In all of these cases, we have $k \to i$, since $\mathcal{G}$ is a TDAG. 
Therefore $gh \in G(\mathcal{G})$, as required for $G(\mathcal{G})$ to be a group. Now if $g \in G(\mathcal{G})$ we show that $g^{-1} \in G(\mathcal{G})$. 
We can write $g = D( I - N)$, where $D$ is diagonal with same diagonal entries as $g$ and $N$ is nilpotent with same zero pattern (outside of the diagonal) as $g$. In fact, since the TDAG $\mathcal{G}$ does not contain any path of length $m$, we have $N^m=0$.
Then
\[ g^{-1} = (I + N + N^2 + \dots + N^{m-1})D^{-1} \in G(\mathcal{G}),
\]
since $\supp(N^j) \subseteq \supp(N)$ for $j \geq 1$, as $\mathcal{G}$ is a TDAG.
We have shown that $G(\mathcal{G})$ is a group.
The equality of models follows from  reparametrizing $(I - \Lambda)\T \Omega^{-1} (I-\Lambda)$ by $g\T g$, where $g = \Omega^{-\frac12} (I- \Lambda) \in G(\mathcal{G})$.
\end{proof}

\begin{example}\label{ex:path1}
Let $\mathcal{G}$ be the TDAG $1 \leftarrow 3 \rightarrow 2$. The corresponding group $G(\mathcal{G}) \subseteq \GL_3$ consists of invertible matrices $g$ of the form
\[g = \begin{bmatrix} * & 0 & * \\ 0 & * & * \\ 0 & 0 & * \end{bmatrix}.
\]
 By Proposition~\ref{prop:tdagroup}, we have that the Gaussian graphical model $\mathcal{M}^{\to}_\mathcal{G}$ is a $5$-dimensional linear slice of the cone of symmetric positive definite $3 \times 3$ matrices:
\begin{equation*}
    \hspace*{27mm}
\mathcal{M}^{\to}_\mathcal{G} = \{ g\T g \mid g \in G(\mathcal{G) } \} = \{ \Psi \in \PD_3 \mid \psi_{12}=\psi_{21}=0 \}.
\hspace*{27mm}
\diamondsuit
\end{equation*}
\end{example}

The group $G(\mathcal{G})$ associated to a TDAG $\mathcal{G}$ is Zariski closed and closed under non-zero scalar multiples, but not self-adjoint. 
Hence we are not in
the setting of Theorem~\ref{thm:MLEversusStabilityReal}.
However, we can apply Theorem~\ref{thm:gaussianMLEstable} to derive our main result of this section. 
Since the group $G(\mathcal{G})$ contains orthogonal matrices of determinant $-1$ (e.g. the diagonal matrix whose first entry is $-1$ and all other entries are $1$), Theorem~\ref{thm:gaussianMLEstable} holds for $G(\mathcal{G})_{\mathrm{SL}}^+$ by Remark~\ref{rem:orthogonalMatricesNegativeDet}.

We characterize boundedness of the likelihood and MLE existence, in terms of the stability of a tuple of samples. When the MLE exists generically (i.e., when the number of samples is at least the maximum likelihood threshold), it is known to be generically unique \cite[Section~5.4.1]{Lauritzen}.
We show that the log-likelihood given $Y$ is bounded from above if and only if the MLE given $Y$ exists,
by ruling out the possibility that a tuple can be semistable but not polystable.
We provide an exact condition for the MLE given $Y$ to exist, based on linear dependence of the rows of $Y$.
A \emph{parent} of a node $i$ is a node $j$ with edge $j \to i$ in $\mathcal{G}$. 

\begin{theorem}
\label{thm:tdagnullcone}
Consider a TDAG $\mathcal{G}$ and a tuple of $n$ samples $Y \in \RR^{m 
\times n}$.
If some row of $Y$, corresponding to node $i$, is a linear combination of the rows corresponding to the parents of $i$, then $Y$ is unstable under the action by $G(\mathcal{G})_{\mathrm{SL}}^{+}$, and the likelihood is unbounded from above. Otherwise, $Y$ is polystable and the MLE exists.
\end{theorem}

\begin{remark}
If $Y$ has a row of zeros, it is unstable and the likelihood is unbounded from above.
This satisfies the criterion in the above theorem, because a row of zeros at row $i$ is interpreted as a trivial linear combination, independently of whether node $i$ has parents in $\mathcal{G}$.
\end{remark}

\begin{proof}[Proof of Theorem~\ref{thm:tdagnullcone}]
Without loss of generality, we label the nodes of $\mathcal{G}$ such that $j \to i$ implies $j < i$. Suppose the $i$th node of $\mathcal{G}$ has the first $s$ nodes as parents, and that the $i$th row of $Y$ is a linear combination of the first $s$ rows, 
\[ r_i = \lambda_1 r_1 + \dots + \lambda_s r_s. \]
We show that $Y$ is unstable under $G(\mathcal{G})_{\mathrm{SL}}^{+}$. Let $\varepsilon>0$ and consider the matrix $g_\varepsilon$, which is equal to $\varepsilon I$ except for the $i$th row, which equals
 \begin{equation*}
       (g_\varepsilon)_{ik}= \begin{cases} -\varepsilon^{-(m-1)}\lambda_k & k = 1,\dots,s \\
        \varepsilon^{-(m-1)} & k = i \\ 
        0 & \text{otherwise}.\end{cases} 
       \end{equation*}
We have that $g_\varepsilon \in G(\mathcal{G})_{\mathrm{SL}}^{+}$, since $\det(g_\varepsilon) = 1$ and there are non-zero off-diagonal entries only when $j \to i$. Moreover, the $i$th row of $g_\varepsilon Y$ is the zero vector. Letting $\varepsilon \to 0$ we have that $g_\varepsilon Y \to 0$, so   we conclude that $Y$ is unstable.    
The log-likelihood is unbounded from above, by Theorem~\ref{thm:gaussianMLEstable}.

For the second claim, let $Y$ be such that no row is a linear combination of the rows corresponding to its parents. We show by induction on $m$ that $Y$ is polystable.
This implies that the MLE given $Y$ exists, by Theorem~\ref{thm:gaussianMLEstable}.
If $m=1$, then $G(\mathcal{G})_{\mathrm{SL}}^{+}=\{  1 \}$ and $Y$ is a single non-zero row, 
so the statement holds.
Now for the induction step, $m > 1$,
we assume the claim holds for TDAGs with $m-1$ nodes.

We prove that the orbit $G(\mathcal{G})_{\mathrm{SL}}^{+} \cdot Y$ is closed and hence $Y$ is polystable.
For this, let $Y_0$ be an element of the orbit closure of $Y$.
Then there exists $g_\varepsilon \in G(\mathcal{G})_{\mathrm{SL}}^{+}$ with
$g_\varepsilon Y \to Y_0$ as $\varepsilon \to 0$.
We may assume without loss of generality that $\alpha_\varepsilon := (g_\varepsilon)_{mm} > 0$,
by using an appropriate subsequence of the sequence $(g_\varepsilon)$ and 
multiplying the last row and another row of both $g_\varepsilon$ and $Y_0$ by $-1$ if needed.
Let $g'_\varepsilon$ be obtained from $g_\varepsilon$ by dropping the last row and column and multiplying by $\alpha_\varepsilon^{\nicefrac{1}{m-1}}$. 
Then $g'_\varepsilon \in G(\mathcal{G}')_{\mathrm{SL}}^{+}$, where the TDAG $\mathcal{G}'$ is obtained from $\mathcal{G}$ by removing the last node (and all edges pointing to it).
Similarly, let $Y'$ and $Y_0'$ be obtained from $Y$ and $Y_0$, respectively, by dropping the last row. 
Since $g_\varepsilon Y \to Y_0$, we have that 
\begin{equation}
\label{eq:alphaConvergence}
    \alpha_\varepsilon^{\nicefrac{-1}{m-1}} g'_\varepsilon Y' \to Y'_0 \, \text{ as } \, \varepsilon \to 0.
\end{equation}
Since no row of $Y$ is a linear combination of the rows corresponding to its parents, the same is true of $Y'$, and we apply the induction hypothesis to see that $Y'$ is polystable.
We will use this to construct a group element that sends $Y$ to $Y_0$.

Without loss of generality, assume $m-s,\dots,m-1$ are the parents of the last node~$m$. Then the last row of $g_\varepsilon$ is
$[ 0, \dots, 0, \beta_{s\varepsilon}, \dots, \beta_{1\varepsilon}, \alpha_\varepsilon ]$
and therefore the last row of $g_\varepsilon Y$ is
\begin{equation*}
    \beta_{s\varepsilon} r_{m-s} + \dots + \beta_{1\varepsilon} r_{m-1} + \alpha_\varepsilon r_{m}.
\end{equation*}
Now, let $t \leq s$ be the dimension of the vector space spanned by $r_{m-s}, \ldots, r_{m-1}$ and assume, without loss of generality, that the rows $r_{m-t},\dots,r_{m-1}$ are linearly independent.
Then we can rewrite the last row of $g_\varepsilon Y$ as
\begin{equation}\label{eq:convGamma}
        \gamma_{t\varepsilon} r_{m-t}  + \dots + \gamma_{1\varepsilon} r_{m-1}  +  \alpha_\varepsilon r_{m}
    \end{equation}
for some $\gamma_{i\varepsilon} \in \RR$.
Since $r_m$ is not a linear combination of its parents, the rows $r_{m-t}, \ldots, r_m$ are linearly independent, i.e. the matrix $M \in \RR^{(t+1) \times n}$ formed by these rows has rank $t+1$.
Thus, any standard basis vector in $\RR^{t+1}$ can be expressed as a linear combination of the columns of $M$. 
Applying these linear combinations to \eqref{eq:convGamma}, which is the last row of $g_\varepsilon Y$ and converges to the last row of $Y_0$, we conclude convergence of each $\gamma_{i\varepsilon}$ ($1\leq i \leq t$) and of $\alpha_{\varepsilon}$ as $\varepsilon \to 0$. We denote the corresponding limits by $\gamma_{i0} \in \RR$ and $\alpha_0 \geq 0$ respectively.

If $\alpha_0 = 0$,  we get from~\eqref{eq:alphaConvergence} that $g'_\varepsilon Y'= \alpha_\varepsilon^{\nicefrac{1}{m-1}} (\alpha_\varepsilon^{\nicefrac{-1}{m-1}} g'_\varepsilon Y') \to 0$ as $\varepsilon \to 0$. So $Y'$ is unstable, in particular not polystable, which contradicts the induction hypothesis.

Therefore, $\alpha_0 > 0$ and we have $g'_\varepsilon Y'= \alpha_\varepsilon^{\nicefrac{1}{m-1}} (\alpha_\varepsilon^{\nicefrac{-1}{m-1}} g'_\varepsilon Y') \to \alpha_0^{\nicefrac{1}{m-1}} Y'_0$ as $\varepsilon \to 0$.
Applying the induction hypothesis to $Y'$, we obtain that $\alpha_0^{\nicefrac{1}{m-1}} Y'_0$ lies in the orbit of $Y'$ under the action by $G(\mathcal{G}')_{\mathrm{SL}}^{+}$.
This means there exists $h' \in G(\mathcal{G}')_{\mathrm{SL}}^{+}$ such that $\alpha_0^{\nicefrac{1}{m-1}} Y'_0 = h'Y'$ 
and therefore
\[ h := \begin{bmatrix} \alpha_0^{\nicefrac{-1}{m-1}} h' & 0 \\ 0 \cdots 0 \, \gamma_{t0} \cdots \gamma_{10} & \alpha_0 \end{bmatrix}  \in G(\mathcal{G})_{\mathrm{SL}}^{+}
\]
satisfies $hY = Y_0$ as desired.
\end{proof}

Our approach characterizes MLE existence for any tuple $Y$, not just generic existence. 
We derive an immediate corollary for generic tuples, 
regarding the maximum likelihood thresholds $\mlt$ and $\mltb$ defined in Section~\ref{sec:MLEprelim}.
This is known for general DAGs in the graphical models literature, see \cite[Section 5.4.1]{Lauritzen} and \cite[Theorem~1]{drton2019maximum}. The \emph{in-degree} of a DAG $\mathcal{G}$ is the maximum number of parents of any node in $\mathcal{G}$. 

\begin{corollary} \label{cor:tdagmlt}
For the model $\mathcal{M}_\mathcal{G}^{\to}$ of a TDAG $\mathcal{G}$, we have
\[ \mltb(\mathcal{G}) = \mlt(\mathcal{G}) = \text{in-degree}(\mathcal{G}) + 1. \]
\end{corollary}
\begin{proof}
The equivalence of the two maximum likelihood thresholds follows from Theorem~\ref{thm:tdagnullcone}, where we also see that
for the MLE to exist generically we need that every row in a generic matrix of samples $Y \in \mathbb{R}^{m \times n}$ is not  a linear combination of its parent rows. Generic linear independence is guaranteed if and only if the number of columns $n$ is at least the number of rows involved in a node plus its parents.
\end{proof}

\begin{example}\label{ex:path2}
Let $\mathcal{G}$ be the TDAG $1 \leftarrow 3 \rightarrow 2$ from Example \ref{ex:path1}. We apply Theorem~\ref{thm:tdagnullcone} to show when the MLE given a sample matrix $Y \in \RR^{3 \times n}$ exists.
Node 3 has no parents, while nodes 1 and 2 both have the node 3 as their parent. 
Hence the log-likelihood $\ell_Y$ is unbounded from above if the first or second row is a scalar multiple of the third row, or if the third row is zero, and otherwise the MLE given $Y$ exists. 

When $n=1$, the first and second rows are always scalar multiples of the third row, hence the null cone fills the space, and the log-likelihood is always unbounded from above.
With $n=2$ samples, the null cone has two components, with vanishing ideal
\[ \langle y_{11} y_{32} - y_{12} y_{31}  \rangle \cap \langle y_{21} y_{32} - y_{22} y_{31}  \rangle. \]
For generic $Y \in \RR^{3 \times 2}$, these equations do not vanish and the  MLE given $Y$ exists. As in Corollary~\ref{cor:tdagmlt}, the maximum likelihood threshold is $\mlt(\mathcal{G}) = \mltb(\mathcal{G}) = 2$.
\hfill\exSymbol
\end{example}

In the previous example the null cone is Zariski closed, but this is not always the case. We now give a precise criterion for when this happens.
An \emph{unshielded collider} of a directed graph $\mathcal{G}$ is a subgraph $j \to i \leftarrow k$ with no edge between $j$ and $k$. 

\begin{corollary}\label{cor:tdagNC}
Let $\mathcal{G}$ be a TDAG, and consider the action of $G(\mathcal{G})_{\mathrm{SL}}^{+}$ on tuples of $n$ samples.
The irreducible components of the Zariski closure of the null cone are determinantal varieties:
each component is
defined by the maximal minors of the submatrix whose rows are a childless node and its parents.
For $n \geq \mlt(\mathcal{G})$,
the null cone is Zariski closed if and only if $\mathcal{G}$ has no unshielded colliders.
\end{corollary}

\begin{proof}
By Theorem~\ref{thm:tdagnullcone}, the null cone is the union 
\begin{equation}
\label{eq:tdagNullcone}
    \bigcup_{i = 1}^m \mathcal{L}(i),
\end{equation}
where $\mathcal{L}(i)$ consists of all $m \times n$ matrices whose $i$th row is a linear combination of rows corresponding to the parents of node $i$.
Since the closure of a finite union is the union of the closures, the Zariski closure of~\eqref{eq:tdagNullcone} is a union of determinantal varieties~$\overline{\mathcal{L}(i)}^{Z}$, 
each given by the maximal minors of the submatrix formed by node $i$ and its parents.
If node $i$ has a child $c$, then $\overline{\mathcal{L}(i)}^{Z} \subset \overline{\mathcal{L}(c)}^{Z}$, because of the transitivity of $\mathcal{G}$. The first part of the assertion follows.

For the second part, 
we assume without loss of generality that the labels are ordered such that $j \to i$ implies $j<i$.
We start by assuming that $\mathcal{G}$ has no unshielded colliders.
Let $Y$ be a matrix in the Zariski closure of the null cone, i.e.
there is some node $i$ with parents $p_1 < \ldots < p_s$ such that the corresponding $s+1$ rows $r_i, r_{p_1}, \ldots, r_{p_s}$ of $Y$ are linearly dependent.
So there is a nontrivial linear combination $\lambda_1 r_{p_1} + \ldots + \lambda_s r_{p_s} + \lambda_{s+1} r_i = 0$.
We pick the largest index $\ell$ such that $\lambda_\ell \neq 0$.
If $\ell = s+1$, the $i$th row is a linear combination of its parents, and $Y$ is in the null cone.
Otherwise, the row $r_{p_\ell}$ is a linear combination of $r_{p_1}, \ldots,r_{p_{\ell-1}}$. We claim that these are all parents of $p_l$, and therefore that $Y$ is in the null cone. Indeed, if some $p_j$ for $1 \leq j \leq \ell - 1$ was not a parent of $p_\ell$, we would have the unshielded collider $p_j \to i \leftarrow p_\ell$.

Conversely, we assume that some node $i$ has two parents $j<k$ that are not connected.
If $i$ has several such pairs of parents, we consider a pair $(j,k)$ such that $k$ is minimal.
This assures that every parent $p$ of $k$ must also be a parent of $j$. Indeed, by transitivity of the DAG $\mathcal{G}$, we have that $p \to i$ and that $j \not\to p$ (since $j \not\to k$). Moreover, by minimality of $k$, it cannot be that there is no edge between $p$ and $j$, so $p \to j$.

We will now construct a matrix $Y$ which is not in the null cone but in its Zariski closure.
We assign the rows $1, \ldots, m$ in order, according to the following rules.
Each row, except for $k$, is assigned so that it is linearly independent of its parents.
We note that this is possible due to $n \geq \mlt(\mathcal{G})$.
In particular, the $j$th row is assigned such that it is linearly independent of its parents, which include the parents of $k$ as observed above.
We pick the $k$th row equal to the $j$th row.
Since now the parents $j$ and $k$ of $i$ are linearly dependent, we see that the matrix $Y$ is in the Zariski closure of the null cone.
However, by our construction, no node in $\mathcal{G}$ is a linear combination of its parents, so $Y$ does not lie in the null cone.
\end{proof}

\begin{example}
\label{ex:nullconeNotZariskiClosed}
Let $\mathcal{G}$ be the TDAG $1 \to 3 \leftarrow 2$, with an unshielded collider. The corresponding group $G(\mathcal{G})$ consists of invertible matrices 
\[ g = \begin{bmatrix} * & 0 & 0 \\ 0 & * & 0 \\ * & * & * \end{bmatrix}.
\]
This is the transpose of the group in Examples~\ref{ex:path1} and~\ref{ex:path2}, but we observe differences between the two models. Since node $3$ has the nodes $1$ and $2$ as parents, Corollary \ref{cor:tdagmlt} tells us that $\mlt(\mathcal{G})= \mltb(\mathcal{G}) = 2 + 1 = 3$ (as opposed to $\mlt=2$ in Example \ref{ex:path2}). 

The null cone is not Zariski closed for $n \geq 3$,
by Corollary~\ref{cor:tdagNC}.
Note that the  Zariski closure of the null cone when $n=3$ is generated by the single equation $\det(Y)$.
We see that the null cone is also not closed for $n=2$,
using Theorem~\ref{thm:tdagnullcone}. Here, row 3 is generically a linear combination of rows 1 and 2, and hence the Zariski closure of the null cone fills the space of tuples. However, for special choices of tuple $Y$, the MLE does exist. For example, let
    \[ Y = \begin{bmatrix} 1 & 0 \\ 1 & 0 \\ 0 & 1 \end{bmatrix} . \]
Rows 1 and 2 are non-zero, and row 3 is not a linear
combination of rows 1 and 2, hence the MLE given $Y$ exists. Since $Y$ is of minimal norm in its orbit, one MLE is $2 I_3$, where $\lambda=2$ minimizes $\frac32 \lambda - 3 \log(\lambda)$, see Proposition~\ref{prop:twoStepGeneral}.
In fact, there are infinitely many MLEs, as follows. For any $g$ in the stabilizer of $Y$ the vector $g \cdot Y$
is also of minimal norm in the orbit. 
Then $\lambda g\T g$ is also an MLE given $Y$, where $\lambda = 2$ as before.
The stabilizer is
\[ 
    \left\lbrace \begin{bmatrix} 1 & 0 & 0 \\ 0 & 1 & 0 \\ t & -t & 1 \end{bmatrix} \colon t \in \RR \right\rbrace \quad
\text{thus} \quad
2 I_3 + 2t \begin{bmatrix} t & -t & 1 \\ -t & t & -1 \\ 1 & -1 & 0 \end{bmatrix}, \quad t \in \RR \quad \text{ are also MLEs.} \]
In fact,  we can verify that these are all  MLEs using Proposition~\ref{prop:twoStepGeneral}.
\hfill $\diamondsuit$
\end{example}

We describe the implications of the above results for undirected Gaussian graphical models, i.e. those coming from graphs with undirected edges, see \cite[Chapter 13]{Sullivant}. A Gaussian graphical model on an undirected graph $\mathcal{G}$ is given by all concentration matrices $\Psi$ such that $\psi_{ij}=0$ whenever the edge $i - j$ is missing from $\mathcal{G}$. A natural question is to determine which undirected Gaussian graphical models are Gaussian group models, i.e. of the form $\mathcal{M}_G$ for some group $G \subseteq \GL_m$. For instance, note that the undirected model corresponding to $1 - 3 - 2$ is the same as the directed model from Example \ref{ex:path1}. We argue that any undirected model that is a Gaussian group model is covered by our study of TDAGs.

We first note that the directed model of any TDAG without unshielded colliders equals the undirected model of its underlying undirected graph, see e.g. \cite[Proposition 4.1]{andersson1997markov}.
Conversely, a necessary condition for an undirected graphical model to be a Gaussian group model can be obtained from \cite[Theorem 2.2]{letac2007wishart}: an undirected Gaussian graphical model is a transformation family if and only if the graph $\mathcal{G}$ has neither $4$-cycles nor $4$-chains as induced subgraphs. There are two consequences of these conditions. One is that there is a way to \emph{direct} the edges in $\mathcal{G}$ so that there are no unshielded colliders.
The other consequence is that this can be done in such a way so that the undirected model coincides with the directed model $\mathcal{M}_\mathcal{G}^{\to}$, and the directed graph must be a TDAG, see page 7 of the supplementary material of \cite{draisma2013groups}.
 In summary, we have the following equivalence.

\begin{remark}
\label{rem:undirectedGraphs}
The undirected graphical models that are Gaussian group models are the TDAG models without unshielded colliders. 
They are exactly those models 
whose sets of tuples of $n$ samples with unbounded likelihood are Zariski closed for all $n$,
by Corollary~\ref{cor:tdagNC}.
\end{remark}

\appendix

\section{Connections to representations of quivers}\label{sec:appendix}
We explain how to deduce Theorem~\ref{thm:king} from the general setting in \cite{King} in terms of representations of quivers. We use the Kronecker quiver $Q$ with two vertices and $n$ arrows:
    
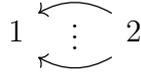
\begin{center}
        \begin{tikzcd}
        1  & 2 \ar[l, shift left = 4pt, bend left] \ar[l, draw=none, "\raisebox{+0.7ex}{\vdots}" description] \ar[l, bend right, shift right = 3pt]
        \end{tikzcd}
    \end{center}
An element $Y$ in $V := (\CC^{m_1 \times m_2})^n$ is a finite dimensional representation of $Q$ with dimension vector $\alpha = (m_1,m_2)$. 
We denote such a representation by $(\CC^{m_1}, \CC^{m_2} ; Y)$.
This identifies $V$ with the space $\mathscr{R}(Q,\alpha)$ from \cite{King}. The left-right action of $G := \GL_{m_1}(\CC) \times \GL_{m_2}(\CC)$ on $V$ by $(g_1,g_2) \cdot (Y_i)_i = (g_1 Y_i g_2^{-1})_i$ is the $\GL(\alpha)$ action on $\mathscr{R}(Q,\alpha)$ from \cite{King}. The difference between this action and our left-right action (with $g_2\T$ rather than $g_2^{-1}$) preserves all stability notions.

We consider two closely related group actions. First, we restrict to $H := \SL_{m_1}(\CC) \times \SL_{m_2}(\CC)$. Second, we consider the action of $G$ on $V \times \CC$~by
    \begin{equation*}
        g \cdot (X,z) := (g \cdot X, \chi_{\theta}^{-1}(g)z), \quad\text{ where } \quad
        \chi_{\theta}^{-1}(g) = [\det(g_1)]^{-m_2} [\det(g_2)]^{m_1},
    \end{equation*}
    for $\theta := (m_2, -m_1)$.
The two actions are related as follows.

\begin{lemma}
\label{lem:RelationToKing}
Fix $Y \in V = (\CC^{m_1 \times m_2})^n$ and $z \in \CC^\times$, and set $\hat{Y} := (Y,1) \in V \times \CC$. Then
    \begin{itemize}\itemsep 3pt
        \item[(a)] $(X,z) \in G \cdot \hat{Y} \quad \Leftrightarrow \quad z^{\frac{1}{m_1 m_2}} X \in H \cdot Y$
        \item[(b)] $(X,z) \in \overline{G \cdot \hat{Y}} \quad \Leftrightarrow \quad z^{\frac{1}{m_1 m_2}} X \in \overline{H \cdot Y}$
        \item[(c)] $\left( \exists \, X \in V  \colon (X,0) \in \overline{G \cdot \hat{Y}} \right) \quad \Leftrightarrow \quad 0 \in \overline{H \cdot Y}$. 
    \end{itemize}
\end{lemma}

\begin{proof}
To prove (a), take $g \in G$ with $(X,z) = g \cdot \hat{Y}$. Then $[\det(g_1)]^{-m_2} [\det(g_2)]^{m_1} = z$ and $g \cdot Y = X$. Set $h := \big( \det(g_1)^{-\frac{1}{m_1}} g_1, \, \det(g_2)^{-\frac{1}{m_2}} g_2 \big) \in H$ to obtain $h \cdot Y = z^{\frac{1}{m_1 m_2}} X$. 
Conversely, given the latter for some $h = (h_1, h_2) \in H$, we define $g := \big( z^{-\frac{1}{m_1 m_2}} h_1, \, h_2 \big)$ to yield $g \cdot \hat{Y} = (X,z)$.
Part~(b) follows from applying (a) to a sequence in the respective orbit that tends to a point in the orbit closure.

For $Y=0$ we have $(0,0) \in \overline{G \cdot \hat{Y}}$ and $0 \in \overline{H \cdot Y}$. It remains to consider $Y \neq 0$. Take $X \in V$ and let $g^{(k)} \in G$ be a sequence such that $g^{(k)} \cdot \hat{Y}$ tends to $(X,0)$ as $k \to \infty$. 
Since $\chi_{\theta}^{-1}(g^{(k)}) \neq 0$ for all $k$, we apply (a) to obtain $
        Y_k := \left[ \chi_{\theta}^{-1}(g^{(k)}) \right]^{\frac{1}{m_1 m_2}} g^{(k)} \cdot Y \in H \cdot Y $
for all~$k$. With $g^{(k)} \cdot \hat{Y} \to (X,0)$ for $k \to \infty$ we conclude that the sequence $Y_k$ tends to $0 \in V$. On the other hand, assume there exist $Y_k \in H \cdot Y$ with $Y_k \to 0$ as $k \to \infty$. Since $Y \neq 0$, we have $Y_k \neq 0$ and hence $c_k := \| Y_k \|^{\frac{m_1 m_2}{2}} \neq 0$ for all $k$. Thus, setting 
    $X_k := c_k^{-\frac{1}{m_1 m_2}} Y_k$
and applying (a) gives $(X_k, c_k) \in G \cdot \hat{Y}$. The latter sequence tends to $(0,0) \in V \times \CC$ by the choice of $c_k$.
\end{proof}

With the help of Lemma~\ref{lem:RelationToKing} we prove Theorem~\ref{thm:king}.

\begin{proof}[Proof of Theorem~\ref{thm:king}]
The equivalence of (a) and (b) is Theorem~\ref{thm:MLEversusStabilityComplex}.
It remains to prove the equivalence of (b) and (c).
Recall that $\theta = (m_2, -m_1)$. By \cite[Proposition~3.1]{King} the matrix tuple $Y = (Y_1,\ldots,Y_n)$ is $\chi_\theta$-stable if and only if the representation $(\CC^{m_1}, \CC^{m_2}; Y)$ is $\theta$-stable. First, we show that the former is equivalent to being stable under the action of $H$. Then we rephrase the latter as the shrunk subspace condition~(c).

Set $\Delta := \lbrace (t I_{m_1}, t I_{m_2}) \mid t \in \CC^\times \rbrace$ and let $G_{\hat{Y}}$ denote the $G$-stabilizer of $\hat{Y} = (Y,1)$. The tuple $Y$ is $\chi_\theta$-stable if and only if the orbit $G \cdot \hat{Y}$ is closed and the group $G_{\hat{Y}}/\Delta$ is finite, by \cite[Lemma~2.2]{King}. The group $G_{\hat{Y}} / \Delta$ is finite if and only if $H_Y$ is finite, since
the group morphism
    \begin{equation*}
        \varphi \colon G_{\hat{Y}} \to H_Y, \quad
        (g_1,g_2) \mapsto \left( \det(g_1)^{-\frac{1}{m_1}} g_1, \, \det(g_2)^{-\frac{1}{m_2}} g_2 \right)
    \end{equation*}
induces an isomorphism $G_{\hat{Y}} / \Delta \cong H_Y$. For $Y \neq 0$, we show that $G \cdot \hat{Y}$ is closed if and only if $H \cdot Y$ is closed, as follows. If $G \cdot \hat{Y}$ is closed and $X \in \overline{H \cdot Y}$, then $(X,1) \in \overline{G \cdot \hat{Y}} = G \cdot \hat{Y}$ using Lemma~\ref{lem:RelationToKing}(b), and hence $X \in H \cdot Y$ by Lemma~\ref{lem:RelationToKing}(a). Conversely, if $H \cdot Y$ is closed with $Y \neq 0$ then $0 \notin \overline{H \cdot Y}$. Thus, Lemma~\ref{lem:RelationToKing}(c) yields $\overline{G \cdot \hat{Y}} \cap \big( V \times \lbrace 0 \rbrace \big) = \emptyset$. Hence any $(X,z) \in \overline{G \cdot \hat{Y}}$ must satisfy $z \in \CC^\times$ and we conclude that
$(X,z) \in G \cdot \hat{Y}$ using Lemma~\ref{lem:RelationToKing}.

For $\theta$-stability, $(\CC^{m_1}, \CC^{m_2}; Y)$ is viewed as an element of the category of finite dimensional representations of the Kronecker quiver $Q$. We note that $\langle \theta, (m_1,m_2) \rangle = 0$ is satisfied by our choice $\theta = (m_2, -m_1)$. We specialize \cite[Definition~1.1]{King} to our representation $(\CC^{m_1}, \CC^{m_2}; Y)$ of the Kronecker quiver $Q$. The representation is $\theta$-semistable if and only if
for all subrepresentations of $(\CC^{m_1}, \CC^{m_2}; Y)$, i.e. all subspaces $V_1 \subseteq \CC^{m_1}$, $V_2 \subseteq \CC^  {m_2}$ such that $Y_i V_2 \subseteq V_1$ for all $i$, we have
    \begin{equation}\label{eq:thetaStable}
        \langle \theta, (\dim V_1, \dim V_2) \rangle = m_2 \dim V_1 - m_1 \dim V_2 \geq 0.
    \end{equation}
The representation $(\CC^{m_1}, \CC^{m_2}; Y)$ is $\theta$-stable if and only if
in addition, the inequality in~\eqref{eq:thetaStable} is strict for all non-zero proper subrepresentations. Here, non-zero means $V_1 \neq 0$ or $V_2 \neq 0$, while proper means $V_1 \subsetneq \CC^{m_1}$ or $V_2 \subsetneq \CC^{m_2}$. Since $V_1 \neq 0$ and $V_2 = 0$ gives strict inequality in \eqref{eq:thetaStable}, it is enough to consider $V_2 \neq 0$. On the other hand, strict inequality in \eqref{eq:thetaStable} holds for all proper subrepresentations satisfying $V_1 \subsetneq \CC^{m_1}$ and $V_2 = \CC^{m_2}$ if and only if there is \emph{no} proper subrepresentation of this form, i.e. if and only if $\mathrm{rank}(Y_1,\ldots,Y_n) = m_1$. Hence, by requiring the latter condition we can restrict to the case $V_2 \subsetneq \CC^{m_2}$. All together, we rephrased the $\theta$-stability of $(\CC^{m_1}, \CC^{m_2}; Y)$ as in the statement.
\end{proof}

\begin{remark}
\label{rem:kingSemistable}
Proposition~3.1 in \cite{King} provides an alternative proof of the complex analog of Theorem~\ref{thm:nullconeLeftRight}, i.e. \cite[Proposition~2.1]{BurginDraisma}.
It  states that $Y$ is $\chi_\theta$-semistable if and only if $(\CC^{m_1}, \CC^{m_2}; Y)$ is $\theta$-semistable. The former holds if and only if 
    \begin{equation*}
        \big( V \times \lbrace 0 \rbrace \big) \cap \overline{G \cdot \hat{Y}} \neq \emptyset,
    \end{equation*}
i.e. if and only if $Y$ is semistable under the action of $H$, by Lemma~\ref{lem:RelationToKing}.
On the other hand, the proof of Theorem~\ref{thm:king} shows that $(\CC^{m_1}, \CC^{m_2}; Y)$ is $\theta$-semistable if and only if \eqref{eq:thetaStable} holds for all subspaces $V_1 \subseteq \CC^{m_1}$, $V_2 \subseteq \CC^{m_2}$ satisfying $Y_i V_2 \subseteq V_1$ for all $i = 1,\ldots,n$.
\end{remark}

\bigskip

{\small
\paragraph{\textbf{Acknowledgements}}
We are grateful to 
Peter B\"urgisser, Mathias Drton, Bernd Sturmfels, and Michael Walter
for fruitful discussions.
We also thank
Jan Draisma,
Visu Makam,
Nikolay Nikolov,
Panagiotis Papazoglou,
Piotr Zwiernik and the anonymous referees
for useful hints and suggestions.
CA was partially supported by the Deutsche Forschungsgemeinschaft (DFG) in the context of the Emmy Noether junior research group KR 4512/1-1.
KK was
partially supported by the Knut and Alice Wallenberg Foundation within their WASP
(Wallenberg AI, Autonomous Systems and Software Program) AI/Math initiative. 
Research of PR is funded by the European Research Council (ERC) under the European’s Horizon 2020 research and innovation programme (grant agreement no. 787840).
}

\bibliographystyle{alpha}
\bibliography{literature}

\bigskip\bigskip

\noindent
\footnotesize {\bf Authors' addresses:}

\smallskip 

\noindent
\noindent Technische Universit\"at M\"unchen, Germany,
\hfill {\tt carlos.amendola@tum.de}

\noindent KTH Royal Institute of Technology, Sweden,
\hfill {\tt kathlen@kth.se}

\noindent  Technische Universit\"at Berlin, Germany,
\hfill {\tt reichenbach@tu-berlin.de}

\noindent  University of Oxford, United Kingdom,
\hfill {\tt seigal@maths.ox.ac.uk}

\end{document}